\documentclass[a4paper,11pt]{amsart}
\usepackage{amsfonts,amsopn}

\usepackage[left=2.5cm, right=2.5cm, width=16cm, height=21cm ]{geometry}

\newtheorem{rem}{Remark}

\newcommand{\uu}{\mathbf{u}}
\newcommand{\uuu}{\mathbf{\hat{u}}}
\newcommand{\eee}{\mathbf{\hat{e}}}
\newcommand{\DD}{\mathbf{D}}
\newcommand{\Sig}{\boldsymbol{\Sigma}}
\newcommand{\TT}{\boldsymbol{\tau}}
\newcommand{\dt}{\delta t}
\newcommand{\nn}{\mathbf{n}}
\newcommand{\sss}{\mathbf{s}}
\newcommand{\tn}{t_n}
\newcommand{\tnn}{t_{n+1}}
\newcommand{\tnm}{t_{n-1}}
\newcommand{\enn}{^{n+1}}
\newcommand{\enm}{^{n-1}}
\newcommand{\en}{^{n}}
\newcommand{\io}{\int _{\Omega}}

\newcommand{\ww}{\mathbf{w}}
\newcommand{\vv}{\mathbf{v}}

\newcommand{\ee}{\mathbf{e}}

\newcommand{\LL}{\mathbf{L}}

\newcommand{\R}{\mathbb{R}}

\newcommand{\N}{\mathbb{N}}

\DeclareMathOperator{\Div}{div}
\DeclareMathOperator{\sign}{sign}
\DeclareMathOperator{\tr}{tr}
\DeclareMathOperator{\L2}{L^2(\Omega)}
\DeclareMathOperator{\Linf}{L^{\infty}(\Omega)}

\DeclareMathOperator{\Hach1}{H^1(\Omega)}

\DeclareMathOperator{\Ll2}{ \textbf{L} ^2(\Omega)}

\DeclareMathOperator{\Haach1}{ \textbf{H} ^1(\Omega)}
\DeclareMathOperator{\Hh10}{ \textbf{H} ^1_0(\Omega)}

\newtheorem{theorem}{Theorem}

\newtheorem{lemma}{Lemma}
\newtheorem{proposition}{Proposition}

\begin{document}

\title[A bi-projection method for incompressible Bingham flows]{A bi-projection method for incompressible Bingham flows with variable density, viscosity and yield stress}
\author{R\'enald Chalayer, Laurent Chupin and Thierry Dubois}
\address{Laboratoire de Math\'ematiques Blaise Pascal, UMR 6620, Universit\'e Clermont Auvergne and CNRS,
          Campus des C\'ezeaux, 3 place Vasarely, TSA 60026 CS 60026, 63178 Aubi\`ere cedex, France}
\email{Renald.Chalayer@uca.fr,Laurent.Chupin@uca.fr,Thierry.Dubois@uca.fr}
\date{\today}


\begin{abstract}
A new numerical scheme for solving incompressible Bingham flows with variable density, plastic viscosity and yield stress is proposed.
The mathematical and computational difficulties due to the non-differentiable definition of the stress tensor in the plug regions, \textit{i.e.}
where the strain-rate tensor vanishes, is overcome by using a projection formulation as in the Uzawa-like method for viscoplastic flows.
This projection definition of the plastic tensor is coupled with a fractional time-stepping scheme designed for Newtonian incompressible flows with
variable density. The plastic tensor is treated implicitly in the first sub-step of the fractional time-stepping scheme and a fixed-point iterative procedure
is used for its computation. A pseudo-time relaxation term is added into the Bingham projection whose effect is to ensure a geometric convergence of the
fixed-point algorithm. This is a key feature of the bi-projection scheme which provides a fast and accurate computation of the plastic tensor.
Stability and error analyses of the bi-projection scheme are provided. 
The use of the discrete divergence-free velocity to convect the density in
the mass conservation equation allows us to derive lower and upper bounds for the discrete density.
The error induced by the pseudo-time relaxation term is controlled by a prescribed numerical parameter so that a first-order estimate of the time
error is derived for the velocity field and the density, as well as the dependent parameters that are the plastic viscosity and the yield stress.
\end{abstract}


\maketitle

\section{Introduction}
Viscoplastic materials are common in many industrial processes, as in food industry with dairy products, chocolate confection, pulp suspensions, in the petroleum industry with drilling mud,
cement, waxy crude oil, and in various geophysical phenomena such flows of slurries, debris and lava. In many cases, mixtures of fluids with different properties occur as sliding of mud in water,
dam break of granular matter in air or water. In food industry, the cleaning of a viscoplastic material from conduits is achieved through the displacement induced by water steam or air.
Also in the process of oil recovery, a flow of gas, steam or even water is used to push and transport heavy and highly viscous fluids.

Viscoplastic fluids flow only if the stress exceeds a threshold, namely the yield stress, otherwise they do not deform and behave like solids.
Typical viscoplastic fluid flows exhibit both yielded and unyielded regions. The rigid structures, where the deformation rates vanish, may either be plugs moving as solid bodies or frozen zones,
\textit{i.e.} where the flow is at rest (no velocity). The most commonly used yield stress model to account for this peculiar behaviour is the Bingham model. In the constitutive law, the stress tensor
is not prescribed until the yield stress is reached while, above the threshold, it is proportional to the strain-rate tensor.
This singular change of rheological properties induces mathematical difficulties. From a computational point of view, reproducing accurately the yield surfaces separating yielded and unyielded regions
is challenging (see \cite{Saramito_Wachs_2017} for a recent review on numerical simulation of viscoplastic flows).
Two approaches have been mainly used both in the fields of theoretical studies and numerical simulations, namely regularization methods (see for instance~\cite{Bercovier-Engelman}
and~\cite{Papanastasiou}) and optimization techniques relying on the theory of variational inequalities (see~\cite{DuvautLions} and~\cite{GLT_1981}).
Note that theoretical analyses concerning the questions of existence, uniqueness and regularity of solutions have been investigated in~\cite{DuvautLions} and more recently
in~\cite{Fuchs1998,Fuchs2000}. 
Due to their simplicity regularization methods are appealing~: the singular Bingham law is replaced by a regularized form making the flow viscous everywhere with a large but finite viscosity in
the unyielded regions. The main drawback of this approach is the difficulty to accurately locate the yield surfaces which are not accounted for by the regularized model.
A review for this approach can be found in~\cite{Frigaard_Nouar2005}.

The second approach used to overcome the difficulty due to the non-differentiability of the definition of the stress tensor is based on the variational formulation of the Navier-Stokes equations
which leads to a saddle-point problem. This optimization problem can be solved by using the Uzawa-like method or the augmented Lagrangian method (see~\cite{OverviewBingham}
and the references therein). Both algorithms allow for accurate but rather expansive numerical simulations when a fine description of the plug regions and the yield surfaces is required
(see~\cite{Muravleva,Marly2017}). Efforts have been recently done to accelerate  the convergence of the Augmented Lagrangian algorithm (see~\cite{Treskatis2016,Saramito2016}).
 In~\cite{Saramito2016}, a modified Newton method achieving superlinear convergence is proposed.
The Uzawa-like method (see~\cite{OverviewBingham,Muravleva}) relies on a (pointwise) projection operator for the computation of the extra (plastic part) stress tensor. In~\cite{OverviewBingham}, the authors
suggested to add a pseudo-time relaxation term in the Bingham projection operator in order to increase the convergence rate of the fixed-point algorithm used to compute the plastic tensor.
In~\cite{ChupinDubois}, this approach was coupled with a projection scheme~\cite{Chorin,Temam} for the time discretization of the Navier-Stokes equations. The resulting bi-projection scheme was analyzed in terms
of stability study and error estimates.
More precisely, it was shown in~\cite{ChupinDubois} that the pseudo-relaxation term in the Bingham projection does not deteriorate the first-order accuracy of the time discretization scheme.
Indeed, the error is of the order $\sqrt{\delta t (1+\theta)}$ which is of first order if the prescribed relaxation parameter $\theta$ is of the order of the time step $\delta t$.
Through various numerical simulations in standard configurations, the bi-projection scheme was shown in~\cite{ChupinDubois} to be able to reproduce the characteristic property of Bingham
flows to return to rest in finite time, to accurately predict the plug regions, and to be efficient at both large Bingham and high Reynolds numbers.

The main contribution of the present paper is to extend the numerical scheme proposed in~\cite{ChupinDubois} to incompressible Bingham flows with variable (in space and time)
density, plastic viscosity and yield stress. The key feature of the projection method for the time discretization of the Navier-Stokes equations is to decouple the computations of the velocity field
and the pressure. In a first step, a predicted velocity which is not divergence free is computed from which the gradient of a scalar function (a pseudo-pressure) is then subtracted resulting in
a solenoidal velocity field. The Helmholtz decomposition is invoked in the second step so that the (pseudo-)pressure is solution of a Poisson equation. 
While introduced in the pioneering works of Chorin~\cite{Chorin} and Temam~\cite{Temam} at the end of the 60's in the case of homogeneous fluid flows, \textit{i.e.} with
constant density and viscosity, projection methods have first been analyzed in terms of error estimates by Shen in~\cite{Shen1order} in the early 90's.
Let us also mention~\cite{GQ-error} for an analysis of a finite element fully discrete version.
Fractional step schemes in the spirit of projection methods have been later on extended to incompressible Navier-Stokes equations with variable density and used in numerical simulations, see for
example~\cite{BellMarcus,Almgren-al_1998}. Concerning the mathematical analysis of projection schemes for variable-density flows, stability results have been proved in~\cite{GuermondQuartapelle,PyoShen3}
and error estimates have been derived in~\cite{GuermondSalgado5} for a fully discrete version based on the finite element method for the spatial discretization.
For density-variable flows, if the projection step is achieved as in the homogeneous case, \textit{i.e.} by using the Helmholtz decomposition, 
a variable coefficient elliptic equation has to be solved at each time step (see~\cite{GuermondSalgado4} for instance) in order to compute the pressure, inducing both mathematical and
computational difficulties. In~\cite{GuermondSalgado4,GuermondSalgado5} an alternative approach based on the interpretation of projection methods as penalty methods is used.
As a result, the fractional time-stepping methods proposed in~\cite{GuermondSalgado4,GuermondSalgado5} necessitate the resolution of only one Poisson equation per time step and are therefore
more efficient from a computational point of view. 

The aim of this paper is to make use of these recent developments to propose a new time discretization scheme for solving the equations describing the motion of Bingham fluids with variable
 density, plastic viscosity and yield stress. Unlike in \cite{GuermondSalgado4,GuermondSalgado5}, we consider the temporal semi-discrete equations. Also, as a main difference, the divergence-free velocity field is used as
convective velocity in the discrete version of the mass conservation equation allowing to derive lower and upper bounds as well as error estimates on the density. 
Concerning the treatment of the plastic tensor, a fixed-point algorithm is invoked to solve the Bingham projection, in the spirit of \cite{Muravleva}. As in~\cite{ChupinDubois}, a pseudo-time
relaxation term is added in the computation of the plastic tensor through the Bingham projection in order to provide a geometric convergence of this algorithm.
The objective of this paper is to perform stability study and error analyses of the proposed bi-projection scheme. To the best of our knowledge, no stability and convergence analysis of
projection schemes for incompressible variable density Bingham flows have been done yet. As in the case of homogeneous Bingham flows, the bi-projection scheme is proved
to be stable and first order accurate as long as the relaxation parameter is taken of the order of the time step.

The paper is organized as follows. In Section 2, the mathematical formulation for a Bingham model with variable coefficients is introduced. A projection formulation for the plastic part of the
stress tensor which will be suitable for the construction of our scheme is provided. In Section 3, some additional notations are introduced and preliminary results, that will be useful in the
following sections, are given. In Section 4, the bi-projection scheme, as a time approximation of the continuous model, is shown to be well-posed and bounds on the density,
and the dependent coefficients, that are the viscosity and yield stress, are obtained. Sections 5 and 6 are respectively devoted to the stability and error (convergence) analyses of the scheme. 

\section{The model of a Bingham viscoplastic flow with variable density}

\subsection{The mathematical model}

Let $T>0$ be a positive real number and $\Omega$ a bounded domain of $\mathbb{R}^3$; we denote by $\Gamma$ the boundary of $\Omega$. The isothermal flow of an incompressible viscoplastic medium with variable density is modeled by the following equations, satisfied by the velocity $\uu$, the density $\rho$ and the pressure $p$,
\begin{equation} \label{ModelContinuEq1}
\begin{cases}
\partial _t \rho + \uu \cdot \boldsymbol{\nabla} \rho = 0,\\
\rho \big ( \partial _t \uu + \uu \cdot \boldsymbol{\nabla} \uu \big ) + \boldsymbol{\nabla} p = \Div \TT, \\
\Div \uu =0,
\end{cases}
\end{equation}
in $(0,T) \times \Omega$. The deviatoric stress tensor $\TT$ is defined by the relation
\begin{equation} \label{ModelContinuEq2}
\TT = 2 \mu (\rho) \DD \uu + \alpha ( \rho ) \Sig,
\end{equation}
with $\DD \uu = \frac{1}{2} \big ( \boldsymbol{\nabla} \uu +{}^{T}\!\boldsymbol{\nabla}\uu\big)$ the strain-rate tensor and $\Sig$ is the extra (plastic) part of the stress tensor.
The plastic viscosity $\mu$ and the yield stress $\alpha$ are functions of the density $\rho$ of class $\mathcal{C} ^1$ and are 
assumed to be respectively positive and non-negative.
We denote by
\begin{equation*}
  \mathcal{Q}_0 = \lbrace (t,\boldsymbol{x}) \in (0,T) \times \Omega ; \quad \DD \uu (t,\boldsymbol{x})= \boldsymbol{0}  \rbrace
\end{equation*}
the sub-domain of $(0,T)\times\Omega$ where the strain-rate tensor vanishes. The plastic tensor $\Sig \in \mathbb{R}^{3 \times 3}$ is defined by (see \cite{DuvautLions} for instance)
\begin{equation} \label{PartiePlastiqueContinueD1}
\begin{cases} \displaystyle
\Sig(t,\boldsymbol{x}) = \frac{\DD \uu (t,\boldsymbol{x})}{\vert \DD \uu (t,\boldsymbol{x}) \vert } \quad \text{for} \quad (t,\boldsymbol{x}) \in \big ( (0,T) \times \Omega \big ) \setminus \mathcal{Q}_0 , \\[0,25cm]
\vert \Sig (t,\boldsymbol{x}) \vert \leq 1 , \quad {}^{T} \! \Sig (t,\boldsymbol{x}) = \Sig (t,\boldsymbol{x}) , \quad \tr \Sig (t,\boldsymbol{x}) =0 \quad  \text{for} \quad (t,\boldsymbol{x}) \in \mathcal{Q}_0,
\end{cases}
\end{equation}
where, for all $\boldsymbol{\lambda} \in \mathbb{R}^{3 \times 3}$, we denote by
\begin{equation} \label{2nd-Inv}
  \vert \boldsymbol{\lambda} \vert^2 = \frac{1}{2} \tr \big ( {}^{T} \! \boldsymbol{\lambda} \boldsymbol{\lambda} \big ) 
\end{equation}
its second invariant (see~\cite{MR3205441}).
The system \eqref{ModelContinuEq1}-\eqref{ModelContinuEq2}-\eqref{PartiePlastiqueContinueD1} is supplemented with the following initial and boundary conditions
\begin{equation} \label{CondIniBoundModCont}
\rho|_{t=0} = \rho_0 , \quad  \uu |_{t=0} =  \uu _0,  \quad  \uu |_{\Gamma} =0.
\end{equation}
Finally, we assume there exists $\rho_1 , \rho_2 >0$,
\begin{equation} \label{CondIniBoundModCont2}
\begin{aligned} 
\rho_1 \leq \rho _0 \leq \rho_2, \quad \text{\textit{a.e.} in } \Omega, \quad \text{and} \quad  \uu_0 \in \Hh10 \quad \text{with} \quad \Div \uu _0 = 0.
\end{aligned}
\end{equation}
If the velocity field $\mathbf{u}$ solution of~\eqref{ModelContinuEq1} is sufficiently regular (see~\cite{boyerfabrie}), as $\rho$ is solution of a transport equation, $\mathbf{u}$ is divergence free and the initial density satisfies~\eqref{CondIniBoundModCont2}, we have
\begin{equation} \label{DensityBounds}
    \rho_1 \le \rho(t,\mathbf{x}) \le \rho_2 \quad\textrm{\textit{a.e.} in } (0,T)\times\Omega.
\end{equation}
This ensures that, almost everywhere in $(0,T)\times\Omega$, we have
\begin{equation} \label{VarCoeff}
  \begin{aligned}
     & \exists \,\mu_1, \mu_2 > 0, \quad \mu_1 \leq \mu (\rho(t,\mathbf{x})) \leq \mu_2,  \\
     & \exists \,\alpha_1, \alpha_2 \ge 0,  \quad \alpha_1 \leq \alpha (\rho(t,\mathbf{x})) \leq \alpha_2.
  \end{aligned}
\end{equation} 
In the Newtonian case, such regularity results can be obtained if the data are small enough (see for instance~\cite{Huang2015} and the references therein). For yield stress fluids,
we assume that similar regularity results can be obtained. In the rest of the paper, we assume that solutions of~\eqref{ModelContinuEq1} satisfying~\eqref{DensityBounds} exists so that~\eqref{VarCoeff} is satisfied.
%



\subsection{A projection formulation}

We now introduce a projection formulation for the plastic tensor $\Sig$, which is more suited for the construction of our numerical scheme.
\begin{proposition} \label{prop:PropProj}
For all $ \ell >0$, the condition \eqref{PartiePlastiqueContinueD1} is equivalent to the relation
\begin{equation} \label{PartiePlastiqueContinueD2}
\Sig (t,\boldsymbol{x}) = \mathbb{P} \big ( \Sig (t,\boldsymbol{x}) + \ell \DD \uu (t,\boldsymbol{x}) \big ),
\end{equation}
where $\mathbb{P}$ is the projection operator on the closed convex set defined by
\begin{equation*}
   \Lambda = \lbrace \boldsymbol{\lambda} \in \mathbb{R}^{3 \times 3} ; \quad \vert \boldsymbol{\lambda} \vert \leq 1 ,
   \quad {}^{T} \! \boldsymbol{\lambda} = \boldsymbol{\lambda} , \quad \tr \boldsymbol{\lambda} = 0\rbrace.
\end{equation*}
\end{proposition}
\begin{proof}
If $\DD \uu (t,\boldsymbol{x})=0$, the equivalence is obvious as we have in both cases $\Sig (t,\boldsymbol{x}) \in \Lambda$. \\
We now assume that $\DD \uu (t,\boldsymbol{x}) \neq 0.$ An explicit expression of the projection is given by
\begin{equation} \label{ProjBouleUnitHilbert} \displaystyle
\mathbb{P}(\boldsymbol{\lambda})=
\begin{cases}
\boldsymbol{\lambda} \quad & \text{if} \quad \vert \boldsymbol{\lambda} \vert \leq 1, \\
\frac{\boldsymbol{\lambda}}{\vert \boldsymbol{\lambda} \vert} \quad & \text{if} \quad \vert \boldsymbol{\lambda} \vert > 1.
\end{cases}
\end{equation}
Let $\ell>0$ be a positive real number. If \eqref{PartiePlastiqueContinueD1} is satisfied, we have $\displaystyle \Sig(t,\boldsymbol{x})=\frac{\DD \uu (t,\boldsymbol{x})}{\vert \DD \uu (t,\boldsymbol{x}) \vert }$. Therefore,
\begin{equation*}
  \Big \vert \Sig(t,\boldsymbol{x}) + \ell \DD \uu (t,\boldsymbol{x}) \Big \vert =  1 + \ell \big \vert \DD \uu (t,\boldsymbol{x}) \big \vert > 1 ,
\end{equation*}
so that, according to \eqref{ProjBouleUnitHilbert}, we have
\begin{equation*}
  \mathbb{P} \big ( \Sig (t,\boldsymbol{x}) + \ell \DD \uu (t,\boldsymbol{x}) \big ) = \frac{\Sig (t,\boldsymbol{x}) + \ell \DD \uu (t,\boldsymbol{x})}{\vert \Sig (t,\boldsymbol{x}) + \ell \DD \uu (t,\boldsymbol{x}) \vert} = \frac{\DD \uu (t,\boldsymbol{x}) }{\vert \DD \uu (t,\boldsymbol{x}) \vert} = \Sig (t,\boldsymbol{x}),
\end{equation*}
hence, \eqref{PartiePlastiqueContinueD2} follows.\\
Let us now assume that \eqref{PartiePlastiqueContinueD2} is fulfilled. We have $\vert \Sig(t,\boldsymbol{x}) + \ell \DD \uu (t,\boldsymbol{x}) \vert > 1$ (otherwise, $\ell=0$ due to \eqref{ProjBouleUnitHilbert}).
Therefore, according to \eqref{ProjBouleUnitHilbert} we have
\begin{equation*}
  \Sig (t,\boldsymbol{x}) = \frac{\Sig (t,\boldsymbol{x}) + \ell \DD \uu (t,\boldsymbol{x})}{\vert \Sig (t,\boldsymbol{x}) + \ell \DD \uu (t,\boldsymbol{x}) \vert},
\end{equation*}
so that $\vert \Sig (t,\boldsymbol{x}) \vert = 1 $. We also deduce that
\begin{equation*}
  \Sig (t,\boldsymbol{x}) \Big ( \vert \Sig(t,\boldsymbol{x}) + \ell \DD \uu (t,\boldsymbol{x}) \vert -1 \Big ) = \ell \DD \uu (t,\boldsymbol{x}),
\end{equation*}
so that $ \displaystyle \vert \DD \uu (t,\boldsymbol{x}) \vert = \frac{\vert \Sig(t,\boldsymbol{x}) + \ell \DD \uu (t,\boldsymbol{x}) \vert -1}{\ell}.$ Then $ \displaystyle\Sig (t,\boldsymbol{x})=\frac{\DD\uu (t,\boldsymbol{x})}{\vert \DD \uu (t,\boldsymbol{x}) \vert}$ which concludes the proof.   
\end{proof}

\begin{rem} When $ \DD \uu$ is only Lebesgue-integrable in space, we define $ \tilde{\Sig} $ as above with a representative of $\DD \uu$. Then, we define $\Sig$ as the class of $ \tilde{\Sig} $. Introducing
\begin{equation*}
  \boldsymbol{\Lambda} = \lbrace \boldsymbol{f} \in L^2(\Omega)^{3 \times 3} ; \quad \vert \boldsymbol{f} \vert \leq 1, \quad {}^{T}\!\boldsymbol{f}=\boldsymbol{f}, \quad \tr \boldsymbol{f} =0,  \quad \text{\textit{a.e.} in } \Omega \rbrace,
\end{equation*}
we have $\Sig(t) \in \boldsymbol{\Lambda}$.
\end{rem}

With the help of Proposition \ref{prop:PropProj}, we can rewrite the mathematical model for incompressible visco-plastic flows with variable density, viscosity and yield stress as, for any $\ell>0$,
\begin{equation} \label{ModelContinuFinal}
\begin{cases}
\partial _t \rho + \uu \cdot \boldsymbol{\nabla} \rho = 0,\\
\rho \big ( \partial _t \uu + \uu \cdot \boldsymbol{\nabla} \uu \big ) + \boldsymbol{\nabla} p - 2 \Div \big ( \mu ( \rho ) \DD \uu \big ) = \Div \big ( \alpha ( \rho ) \Sig \big ),\\
\Sig = \mathbb{P} ( \Sig + \ell \DD \uu ), \quad \Div \uu =0, 
\end{cases}
\end{equation}
which is supplemented with the initial and boundary conditions \eqref{CondIniBoundModCont}. Note that for constant density flows, the system~\eqref{ModelContinuFinal} reduces to the mathematical model studied in~\cite{ChupinDubois}.
 
\section{Preliminaries}
We now introduce some of the notations used in the sequel. For two vectors $\uu$ and $\vv$ in $\R ^3 $, we denote by $\uu \cdot \vv$ their inner product, $\uu \cdot \vv = \sum_{1 \leq i \leq 3} u_i v_i $, and by $\vert \cdot \vert$ the associated norm. For two tensors $\boldsymbol{A}$ and $\boldsymbol{B}$ in $\R^{3 \times 3}$, we denote by $\boldsymbol{A}:\boldsymbol{B}$ their inner product, namely
\begin{equation*}
  \boldsymbol{A}:\boldsymbol{B} = \sum_{1 \leq i,j \leq 3} A_{ij} B_{ij}.
\end{equation*}
Let us note that the tensorial norm $\vert \cdot \vert$ defined by \eqref{2nd-Inv} is not induced by the above inner product. \\
We will make use of the standard notations $L^{p}(\Omega)$, $H^k(\Omega)$ and $H^k _0 (\Omega)$ to denote the usual Lebesque and Sobolev spaces over $\Omega$. We denote by $L^2_0(\Omega)$ the subspace of $L^2(\Omega)$ of functions with vanishing mean value. The norm corresponding to $H^k(\Omega)$ will be denoted by $\Vert \cdot \Vert_k$. In particular, we will use $\Vert \cdot \Vert$ to denote the norm in $L^2(\Omega)$ and $(\cdot,\cdot)$ to denote the scalar product in $L^2(\Omega)$. For each space above, we will use bold letters to denote their vectorial counterpart: for example, $\L2^3$ will be denoted $\Ll2$. Also other functional spaces may be considered in the sequel. In such cases, they are always indicated as subscripts: for instance $ \Vert\cdot\Vert_{L^{\infty}(\Omega)}$ denotes the norm associated to the space $ L^{\infty}(\Omega)$.

The incompressibility constraint leads us to consider the following space
\begin{equation*}
  \boldsymbol{\mathcal{H}} = \lbrace \uu \in \Ll2; \quad \Div \uu = 0, \quad \uu \cdot \nn |_{\Gamma} =0 \rbrace,
\end{equation*}
where $\nn$ is the unit outward normal to $\Gamma$. For any sequence $(a_n)_{n \in \N}$, we introduce
\begin{equation*}
  \delta a_n = a_n - a_{n-1} \quad\textrm{and}\quad \delta^2 a_n = \delta ( \delta a_n) = a_n -2 a_{n-1} + a_{n-2} .
\end{equation*}
The following lemma of Gronwall type will be used to derive the error estimates in Section 6. Note that a proof can be found in~\cite{lemmegronwalldiscret}.
\begin{lemma}{(Discrete Gronwall lemma).}\label{lem:LGdiscret}
Let $M \in \mathbb{N}^*$, $\tau$, $B$ and $C$ be non-negative parameters, $(y^n)_n,(h^n)_n,(g^n)_n$ and $(f^n)_n$ be non-negative sequences satisfying, for all $m$ such that $0 \leq m \leq M$,
\begin{equation*}
  y^m + \tau \sum _{n=0} ^m h^n \leq B + \tau \sum _{n=0}^m(g^n y^n + f^n), \quad  \text{with} \quad \tau \sum _{n=0} ^M g^n \leq C.
\end{equation*}
Assume $\tau g^n < 1$ for all $n$ such that $0 \leq n \leq M$ and let
\begin{equation*}
 \sigma_n = (1 - \tau g^n)^{-1} , ~~~~ \sigma = \max _{0 \leq n \leq M} \sigma_n,
\end{equation*}
then, for all $m$ such that $0 \leq m \leq M$, we have
\begin{equation*}
  y^m + \tau \sum _{n=0} ^m h^n \leq exp(\sigma C)\Bigl(B + \tau \sum _{n=0} ^m f^n\Bigr).
\end{equation*}
\end{lemma}
Finally, we define the bilinear form $\tilde{B}$ by
\begin{equation*}
  \tilde{B} (\uu,\vv) = (\uu\cdot\boldsymbol{\nabla})\vv + \frac{\vv}{2} \Div \uu,
\end{equation*}
and the associated trilinear form: $\tilde{b}(\uu,\vv,\ww) = (\tilde{B}(\uu,\vv),\ww).$
It can be easily shown that
\begin{equation} \label{CalculBaseStability}
\tilde{b}( \uu, \vv, \vv ) = 0 , \quad \forall \uu \in \Haach1 \quad \text{with} \quad \uu \cdot \nn |_{\Gamma} = 0 , \quad \forall \vv \in \Haach1 .
\end{equation}
Another orthogonality property will be also useful in the sequel, namely we have
\begin{equation} \label{CalculBaseStability0}
   \int_\Omega \varphi\, \uu\cdot\boldsymbol{\nabla} \varphi\, d\boldsymbol{x} = 0, \quad \forall \uu\in\textbf{L}^3(\Omega) \cap \boldsymbol{\mathcal{H}}, \, \forall \varphi\in H^1(\Omega). 
\end{equation}

\section{The bi-projection scheme}

\subsection{The temporal discretization}
In this section, we propose a scheme to discretize with respect to the time variable the model \eqref{ModelContinuFinal}. Let us introduce some additional notations. 
Let $N>0$ be an integer and $(\tn)_{n \in \lbrace 0,...,N \rbrace }$ a sequence of discrete time levels in $[0,T]$. For the sake of simplicity, we consider a uniform discretization, that is
\begin{equation*}
  \tn = n \dt \qquad \text{with} \quad \dt = \frac{T}{N}.
\end{equation*}
We introduce two additional numerical parameters $r>0$ and $\theta > 0$.

We start with $\rho^0$, $\uu^0 $, $ \uuu ^0 $, $q^0$, $p^0$ and $\Sig ^0$ (initialization step). We choose $\rho^0 = \rho_0$, $\uu ^0 = \uuu ^0 = \uu_0$, $q^0 = 0$, and arbitrary $\Sig ^0 $ and $p^0$. \\
For $n \geq 0$, assuming that $\rho \en $, $ \uu \en $, $\uuu \en$, $q^n$, $p^n$ and $ \Sig \en$ are known, we first compute $\rho\enn $ by solving 
\begin{equation} \label{PhaseEquation}
  \frac{\rho^{n+1}-\rho^n}{\dt} + \uuu^n \cdot \boldsymbol{\nabla} \rho^{n+1} =0,
\end{equation}
and, with the help of \eqref{VarCoeff}, we define $\mu^{n+1}$ and $\alpha ^{n+1}$ by the following relations
\begin{equation} \label{PhaseEquation2}
\mu^{n+1} = \mu (\rho^{n+1} ), \quad \alpha^{n+1} = \alpha (\rho^{n+1} ).
\end{equation}
Next, $\uu^{n+1}$ and $\Sig^{n+1}$ are solutions of the system
\begin{equation} \label{Momentum}
\begin{cases}
\begin{aligned} \displaystyle
& \frac{1}{\dt} \Big [ \frac{1}{2}(\rho^{n+1} + \rho^n) \uu^{n+1}  - \rho ^n \uu^n  \Big ] + \tilde{B}(\rho \enn \uuu \en , \uu \enn) \\
& \phantom{\frac{1}{\dt} \Big [ \frac{1}{2}(\rho^{n+1} + \rho^n)} - \Div ( 2 \mu^{n+1} \DD \uu^{n+1} ) + \boldsymbol{\nabla} ( p^n + q^n ) = \Div( \alpha^{n+1} \Sig^{n+1} ) , \\
&\Sig^{n+1} = \mathbb{P} \big ( \Sig^{n+1} + r \alpha^{n+1} \DD \uu ^{n+1} + \theta (\Sig^{n} - \Sig^{n+1}) \big ),\\
&\uu^{n+1}|_{\Gamma} =0.
\end{aligned} 
\end{cases}
\end{equation}
Finally, we compute $q^{n+1}$ by solving the Poisson equation
\begin{equation} \label{Pression}
\begin{cases}
\Delta q^{n+1} = \displaystyle \frac{\rho_{1}}{\dt} \Div (\uu^{n+1}),\\
\partial _{\nn} q^{n+1} |_{\Gamma} =0,
\end{cases}
\end{equation}
and we write
\begin{align}
& p^{n+1} = p^n + q^{n+1}, \label{IncPression} \\
& \uuu^{n+1} = \uu \enn - \frac{\dt}{\rho_{1}} \boldsymbol{\nabla} q \enn. \label{UProj}
\end{align}
\begin{rem} \label{RemUChap}
Using \eqref{Pression}, \eqref{IncPression} and \eqref{UProj}, we observe that, for all $0 \leq n \leq N,$ $ \uuu \en $ satisfies
\begin{equation*}
   \uuu \en \cdot \nn |_{\Gamma} =0 \quad \text{and} \quad  \Div \uuu \en =0.
\end{equation*}
\end{rem}
Note that a fractional time-stepping method in the spirit of \cite{GuermondSalgado4}, \cite{GuermondSalgado5} is employed to decouple the velocity and the pressure. A major difference between \cite{GuermondSalgado4}, \cite{GuermondSalgado5} and the above algorithm \eqref{PhaseEquation}--\eqref{UProj} resides in the treatment of the mass conservation equation, namely the divergence-free velocity $\uuu^{n}$ is used as the convective velocity in the equation of the density~\eqref{PhaseEquation}. With this approach, lower and upper bounds on the
density $(\rho^n)_{n\ge 0}$ can be derived. More precisely, we will prove that,
for all $0 \leq n \leq N$,
\begin{equation*}
\rho_1 \leq \rho \en \leq \rho_2.
\end{equation*}
In order to resolve the coupling in~\eqref{Momentum} between the velocity~$\uu \enn$ and the plastic tensor~$\Sig \enn$, we advocate a fixed-point algorithm, in the spirit of \cite{Muravleva}, which is 
detailed in Section 4.3. As in~\cite{ChupinDubois}, a pseudo-time relaxation term $\theta(\Sig^{n}-\Sig^{n+1})$ is added in the projection operator defining~$\Sig^{n+1}$ in~\eqref{Momentum}. This additional
term guarantees a geometric convergence of this algorithm. \\
System~\eqref{Pression} is a classical Poisson equation. As an immediate consequence, we have $q\enn\in H^2(\Omega)$ and so $\uuu \enn \in \Haach1$ exist and are uniquely defined.
Finally, let us mention that the coefficient $ r \alpha \enn$ is used, instead of $r$ only, in the Bingham projection defining~$\Sig\enn$. We will see in Section 4.3 that the use of $r \alpha \enn $
is required to derive proper estimates and convergence of the fixed-point algorithm.

\subsection{Discrete maximum principle on the density}
The use of the divergence free velocity in the mass conservation equation \eqref{PhaseEquation} allows us to derive positive upper and lower bounds on the density.
This property will be repeatedly used in the rest of the paper. Let us first establish this result.\\
Let $\epsilon >0 $ be a parameter. In order to properly define solutions of \eqref{PhaseEquation}, we introduce the following regularized form
\begin{equation} \label{Viscosity}
\begin{cases} \displaystyle
\frac{\rho^{n+1}_{\epsilon} - \rho^n}{\dt} + \uuu^n \cdot \boldsymbol{\nabla} \rho^{n+1}_{\epsilon} - \epsilon \Delta \rho \enn _{\epsilon} =0, \\
\boldsymbol{\nabla} \rho \enn _{\epsilon} \cdot \nn   |_{\Gamma} =0.
\end{cases}
\end{equation}
The limit of $\rho^{n+1}_{\epsilon}$ when $\epsilon$ tends to $0$ is called the viscosity solution of \eqref{PhaseEquation} (see~\cite{Evans1980} and \cite{GrandallLions1983}). In the remaining part of the paper,
$\rho\enn $ will always be used to denote the viscosity solution of \eqref{PhaseEquation}.
\begin{proposition} \label{prop:StabImpliciteRho}
If $ \hat{\uu}^n \in \boldsymbol{\mathcal{H}} \cap \mathbf{L}^3(\Omega)$ and $\rho_1 \leq \rho \en \leq \rho_2 \textrm{ almost everywhere in } \Omega$, then there exists a unique viscosity solution satisfying \eqref{PhaseEquation} in the distribution sense.
Moreover, this solution satisfies
\begin{equation*}
  \rho_1 \leq \rho \enn \leq \rho_2,\text{ \textit{a.e.} in } \Omega .
\end{equation*}  
\end{proposition}

\begin{proof} \textit{Step 1: Regularized equation.} Equation \eqref{Viscosity} leads us to consider the variational problem
\begin{equation} \label{FV}
  \begin{aligned}
    & \textrm{Find}\, \rho^{n+1}_{\epsilon} \in\Hach1 \,\textrm{so that, for all }\, \varphi \in \Hach1, \,\textrm{we have} \\
    & \io \rho^{n+1}_{\epsilon} \varphi - \dt \io \rho^{n+1}_{\epsilon} \uuu^n \cdot \boldsymbol{\nabla} \varphi + \dt \epsilon \io \boldsymbol{\nabla} \rho^{n+1}_{\epsilon}  \cdot \boldsymbol{\nabla} \varphi = \io \rho^n \varphi.
  \end{aligned}
\end{equation}
By virtue of the Lax-Milgram theorem (using $\hat{\uu}^n \in \boldsymbol{\mathcal{H}} \cap \mathbf{L}^3(\Omega)$ and the property \eqref{CalculBaseStability0}), there exists a unique
$\rho^{n+1}_{\epsilon} \in~\Hach1 $ solution of \eqref{FV}, which satisfies \eqref{Viscosity} almost everywhere in $\Omega$ and $\Gamma$.  \\

\textit{Step 2: Bound of the regularized solution.} 
We want to prove that $\rho_1 \leq \rho_\epsilon \enn \leq~\rho_2$ almost everywhere in $\Omega$. By choosing $\varphi = (\rho \enn _{\epsilon} - \rho_2 )^+$ in \eqref{FV},
where the $+$ superscript denotes the positive part (similarly a $-$ superscript denotes the negative part), using again \eqref{CalculBaseStability0}, and defining the
$\text{sign}$ function by $\sign(x)=1$ if $x \geq 0$ and $\sign(x)=-1$ otherwise, we obtain
\begin{align*} 
  \io (\rho^{n+1}_{\epsilon} - \rho_2 ) (\rho \enn _{\epsilon} - \rho_2 )^+ + \frac{\dt \epsilon}{2} \io \vert \boldsymbol{\nabla} \rho^{n+1}_{\epsilon} \vert ^2   \big ( 1 + & \sign(\rho \enn _{\epsilon} - \rho_2 ) \big ) \\
  \phantom{\io (\rho^{n+1}_{\epsilon} - \Vert \rho \en \Vert_{\Linf}) (\rho \enn _{\epsilon} }= & \io ( \rho^n - \rho_2) (\rho \enn _{\epsilon} - \rho_2 )^+ .
\end{align*}
Observing that the right-hand side is negative and the second term of the left-hand side is non-negative, we have
\begin{equation*}
  \io (\rho^{n+1}_{\epsilon} - \rho_2) (\rho \enn _{\epsilon} - \rho_2 )^+ \leq 0 .
\end{equation*}
On the other hand, $x x^+$ is non-negative for any $x \in \R $. As a consequence,
\begin{equation*}
  (\rho^{n+1}_{\epsilon} - \rho_2) (\rho \enn _{\epsilon} - \rho_2)^+ = 0 \quad \text{\textit{a.e.} in } \Omega,
\end{equation*}
and then
\begin{equation*}
  \rho^{n+1}_{\epsilon} \leq \rho_2 \quad \textrm{\textit{a.e.} in } \Omega.
\end{equation*}
By choosing $(\rho^{n+1}_{\epsilon} - \rho_1 )^-$ in \eqref{FV}, we similarly obtain $\rho^{n+1}_{\epsilon} \geq \rho_1 $ \textit{a.e.} in $\Omega.$
Finally, $\rho^{n+1}_{\epsilon} \in \Linf $ and
\begin{equation} \label{MaxPhiEpsilon}
  \rho_1 \leq \rho _\varepsilon \enn \leq \rho_2,\quad \text{\textit{a.e.} in } \Omega .
\end{equation}

\textit{Step 3: Passage to the limit on $ \epsilon $ and viscosity solution.} By choosing $ \varphi = \rho^{n+1}_{\epsilon} $ in~\eqref{FV}, we have 
\begin{equation*}
\io \vert \rho^{n+1}_{\epsilon} \vert ^2 + \dt \epsilon \io \vert \boldsymbol{\nabla} \rho^{n+1}_{\epsilon}  \vert ^2 = \io \rho^n \rho^{n+1}_{\epsilon} .
\end{equation*}
By invoking the Cauchy-Schwarz inequality, it follows
\begin{equation*}
\Vert \rho^{n+1}_{\epsilon} \Vert ^2 + \dt \Vert \sqrt{\epsilon} \boldsymbol{\nabla} \rho^{n+1}_{\epsilon}  \Vert ^2 \leq \Vert \rho^n \Vert \Vert \rho^{n+1}_{\epsilon} \Vert.
\end{equation*}
Then, with the help of Young's inequality, we obtain
\begin{equation*}
  \Vert \rho^{n+1}_{\epsilon} \Vert ^2 + 2 \dt \Vert \sqrt{\epsilon} \boldsymbol{\nabla} \rho^{n+1}_{\epsilon} \Vert ^2 \leq \Vert \rho^n \Vert ^2.
\end{equation*}
By using this last inequality and \eqref{MaxPhiEpsilon}, we deduce that $( \rho^{n+1}_{\epsilon})_{\epsilon}$, and $(\sqrt{\epsilon} \boldsymbol{\nabla} \rho^{n+1}_{\epsilon})_{\epsilon}$, are bounded sequences
in respectively $\Linf$ and $\Ll2$. Hence, there exists weakly convergent subsequences
\begin{equation} \label{CVFBrezis}
 \rho^{n+1}_{\epsilon} \overset{\ast}{\underset{\epsilon \to 0}{\rightharpoonup} } \rho^{n+1} \quad \text{in} \quad \Linf, \qquad
 \sqrt{\epsilon} \boldsymbol{\nabla} \rho^{n+1}_{\epsilon} \underset{\epsilon \to 0}{\rightharpoonup}  \psi \quad  \text{in} \quad \Ll2.
\end{equation}
As a consequence of the weak convergence and by using \eqref{MaxPhiEpsilon},
\begin{equation*}
  \rho_1 \leq \rho \enn \leq \rho_2,\quad \text{\textit{a.e.} in } \Omega .
\end{equation*}
Finally, for all $\varphi \in \mathcal{D}(\Omega)$, we have
\begin{align*}
\io \rho^{n+1}_{\epsilon} \varphi \quad & \underset{\epsilon \to 0}{\longrightarrow} \quad \io \rho^{n+1} \varphi, \\
- \dt \io \rho^{n+1}_{\epsilon} \uuu^n \cdot \boldsymbol{\nabla} \varphi  \quad & \underset{\epsilon \to 0}{\longrightarrow} \quad - \dt \io \rho^{n+1} \uuu^n \cdot\boldsymbol{\nabla}\varphi , \\
 \dt \epsilon \io \boldsymbol{\nabla} \rho^{n+1}_{\epsilon}  \cdot \boldsymbol{\nabla} \varphi \quad & \underset{\epsilon \to 0}{\longrightarrow} \quad 0,
\end{align*}
so that
\begin{equation*}
  \io \rho^{n+1} \varphi - \dt \io \rho^{n+1} \uuu^n \cdot \boldsymbol{\nabla}\varphi = \io \rho^n \varphi.
\end{equation*}
We conclude that $\rho\enn$ satisfies the transport equation \eqref{PhaseEquation} in the distributions sense.
\end{proof}

We are now able to establish the following result:

\begin{theorem} \label{thm:StabilityPhi}
If $\uu _0 $ and $\rho_0$ satisfy \eqref{CondIniBoundModCont2}, the sequences $(\uuu \en , \rho \en)_{0 \leq n \leq N} $ solutions of \eqref{PhaseEquation}-\eqref{UProj} satisfy,
for all $n$ such that $0\leq n \leq N$,
\begin{equation*}
  \uuu \en \in \boldsymbol{\mathcal{H}} \cap \Haach1,
\end{equation*}
and
\begin{equation*}
\rho_1 \leq \rho \en \leq \rho_2, \quad \text{\textit{a.e.} in }  \Omega.
\end{equation*}
\end{theorem}
\begin{proof}
We proceed by induction.\\
\textit{Case $n=0$.} From \eqref{CondIniBoundModCont2}, using $ \rho^0 = \rho_0$ and $ \uuu ^0 = \uu _0 $, we infer 
\begin{equation*}
 \rho_{1} \leq \rho ^{0} \leq \rho_{2}, \quad \text{\textit{a.e.} in }  \Omega, \quad \text{and} \quad \uuu ^0 \in \boldsymbol{\mathcal{H}} \cap \Haach1. 
\end{equation*}
\textit{Case n $\geq$ 1.} We assume
\begin{equation*}
\rho_1 \leq \rho ^n \leq \rho_2,\textrm{ \textit{a.e.} in } \Omega, \quad \text{and} \quad \uuu \en \in  \boldsymbol{\mathcal{H}} \cap \Haach1.
\end{equation*}
From Proposition~\ref{prop:StabImpliciteRho}, we have $\rho_1 \leq \rho \enn \leq \rho_2$ almost everywhere.
The existence and uniqueness for the solution $(\uu \enn , \Sig \enn ) \in  \Hh10 \times \boldsymbol{\Lambda} $ to the system \eqref{Momentum} is a consequence of the theory of variational inequalities.
The precise result for this system can be found in \cite{OverviewBingham} (page 47).
Then, by solving \eqref{Pression}, we obtain the existence and uniqueness of $q \enn \in H^2(\Omega)$. As a consequence, $\uuu \enn$ given by \eqref{UProj} lies
in $\boldsymbol{\mathcal{H}} \cap \Haach1 $, which concludes the proof.
\end{proof} 

\begin{rem} \label{Rem:Bounds_Visc_Yield}
Using~\eqref{VarCoeff} and the above result, we observe that, for all $0 \leq n \leq N$,
\begin{equation*}
   \mu _1 \leq \mu\en \leq \mu _2 , \quad \alpha_1 \leq \alpha\en \leq \alpha_2 , \quad \textit{a.e.} \text{ in } \Omega.
\end{equation*}
\end{rem}

\subsection{Practical implementation of the Bingham projection}
The equations~\eqref{Momentum} involve a coupling between $\uu^{n+1}$ and $\Sig^{n+1}$. In order to solve this system in practice and following~\cite{ChupinDubois}, we employ a fixed-point iteration procedure and
proceed as it follows. Let us start with $\Sig^{n,0}= \Sig ^n$. For $k\geq 0$ we assume that $\Sig^{n,k} \in \boldsymbol{\Lambda} $ is known and we compute $\uu ^{n,k} \in \Hh10$ by solving the following elliptic problem
\begin{equation} \label{FixPt1}
  \begin{cases}
    \begin{aligned} \displaystyle
& \frac{1}{\dt} \Big [ \frac{1}{2}(\rho^{n+1} + \rho^n) \uu^{n,k} - \rho ^n \uu^n  \Big ] + \tilde{B}( \rho^{n+1} \uuu^n , \uu^{n,k} )\\
& \phantom{\frac{1}{\dt} \Big [ \frac{1}{2}(\rho^{n+1}}- \Div ( 2 \mu^{n+1} \DD \uu^{n,k} ) + \boldsymbol{\nabla}( p^n + q^n ) = \Div( \alpha^{n+1} \Sig^{n,k} ) , \\
    & \uu^{n,k}|_{\Gamma} =0.
    \end{aligned}
  \end{cases}
\end{equation}
Next, a projection is used to explicitly deduce the extra-stress tensor $\Sig^{n,k+1} \in \boldsymbol{\Lambda}$, namely
\begin{equation} \label{FixPt2}
\Sig^{n,k+1} = \mathbb{P} \big ( \Sig^{n,k} + r \alpha^{n+1} \DD \uu ^{n,k} + \theta (\Sig^{n} - \Sig^{n,k}) \big ).
\end{equation}
The following result establishes that the sequence $( \uu ^{n,k}, \Sig ^{n,k})_k $ converges to the desired solution $(\uu ^{n+1},\Sig^{n+1})$ of the equation \eqref{Momentum}. 
\begin{theorem} \label{thm:PtFixeGaugeUzawa}
If $ 4 \theta + r \frac{(\alpha_2)^2}{\mu_1} \leq 4$, then, for all $0 \leq n \leq N-1,$ the sequence $(\uu ^{n,k},\Sig^{n,k})_k$ tends to $(\uu ^{n+1},\Sig^{n+1})$ when $k$ tends to infinity. Moreover the convergence is geometric with common ratio $1- \theta$.
\end{theorem}
\begin{proof}
We denote $ \overline{\uu}^k = \uu^{n,k} - \uu^{n+1}$ and $ \overline{\Sig}^k = \Sig^{n,k} - \Sig^{n+1}$.
By subtracting \eqref{FixPt1}-\eqref{FixPt2} to \eqref{Momentum}, we obtain
\begin{equation} \label{dem31}
\begin{cases}
\begin{aligned} \displaystyle
& \frac{1}{2 \dt} (\rho^{n+1} + \rho^n) \overline{\uu}^k + \tilde{B}( \rho^{n+1} \uuu^n , \overline{\uu}^k ) - \Div ( 2 \mu^{n+1} \DD \overline{\uu}^k ) = \Div( \alpha^{n+1} \overline{\Sig}^k ) , \\[0.15cm]
& \overline{\uu}^k |_{\Gamma} =0 ,\\[0.15cm]
& \overline{\Sig}^{k+1} = \mathbb{P} \big ( \Sig^{n,k}  + r \alpha^{n+1}  \DD \uu ^{n,k} + \theta (\Sig^{n} - \Sig^{n,k}) \big ) \\
& \phantom{\overline{\Sig}^{k+1} = \mathbb{P} \big ( \Sig^{n,k}  + r \alpha^{n+1}} -  \mathbb{P} \big ( \Sig^{n+1} + r \alpha^{n+1} \DD \uu ^{n+1} + \theta (\Sig^{n} - \Sig^{n+1}) \big ).
\end{aligned}
\end{cases}
\end{equation}
We now take the inner product of the first equation in \eqref{dem31} with $\overline{\uu}^k$ in $\Ll2$. Using \eqref{CalculBaseStability}, we deduce
\begin{equation} \label{dem32}
\frac{1}{\dt} \bigg \Vert \sqrt{\frac{\rho^{n+1} + \rho^n}{2}} \overline{\uu}^k \bigg \Vert ^2 + 4 \Vert \sqrt{ \mu^{n+1} } \DD \overline{\uu}^k \Vert ^2 = - ( \alpha^{n+1} \overline{\Sig}^k, \DD \overline{\uu}^k ).
\end{equation}
From the second equation of \eqref{dem31}, since $\mathbb{P}$ is a projection, we derive
\begin{equation*}
\vert \overline{\Sig}^{k+1} \vert \leq \vert  (1- \theta) \overline{\Sig}^k + r \alpha^{n+1} \DD \overline{\uu} ^{k} \vert.
\end{equation*}
By taking the $\L2$-norm, it follows
\begin{equation} \label{dem33}
\Vert \overline{\Sig}^{k+1} \Vert ^2 \leq (1- \theta)^2 \Vert \overline{\Sig}^k \Vert ^2 + r^2 \Vert \alpha^{n+1} \DD \overline{\uu} ^{k} \Vert ^2 + r(1-\theta)( \alpha^{n+1} \overline{\Sig}^k, \DD \overline{\uu}^k ).
\end{equation}
Combining \eqref{dem32}, \eqref{dem33} and making use of $ \alpha \enn \leq \alpha _2$, $\mu \enn \geq \mu _1 $ and $\frac{\rho \enn + \rho \en}{2} \geq \rho_1$, we obtain
\begin{equation*} 
\begin{aligned}
 \Vert \overline{\Sig}^{k+1} \Vert ^2 + \frac{r \rho_1 (1-\theta)}{\dt} \Vert \overline{\uu}^k  \Vert ^2 +  r \big ( 4 \mu_1 (1-\theta) - r  \alpha_2 ^2 \big ) \Vert \DD \overline{\uu}^k \Vert ^2 \leq (1- \theta)^2 \Vert \overline{\Sig}^k \Vert ^2.
\end{aligned}
\end{equation*}
From the assumption $ 4 \theta + r \frac{\alpha_2 ^2}{\mu_1} \leq 4 $, we deduce
\begin{equation*}
 \Vert \overline{\Sig}^{k} \Vert \leq (1- \theta )^k \Vert \overline{\Sig}^{0} \Vert \quad\textrm{and}\quad
 \Vert \overline{\uu}^k  \Vert \leq \sqrt{\frac{\dt (1 - \theta)}{ r \rho_1 }} (1- \theta )^k \Vert \overline{\Sig}^{0} \Vert
\end{equation*}
which concludes the proof of Theorem~\ref{thm:PtFixeGaugeUzawa}.
\end{proof}
Note that, if $ 4 \theta + r \frac{\alpha_2 ^2}{\mu_1} < 4 $, we also have convergence in $\Haach1 $ for the velocity, namely
\begin{equation*}
  \Vert \DD \overline{\uu}^k  \Vert \leq \frac{1- \theta}{\sqrt{r(4 \mu_1 (1- \theta )- r \alpha ^2 _2 )}} (1- \theta )^k \Vert \overline{\Sig}^{0} \Vert.
\end{equation*}

\section{Stability analysis}
In Theorem~\ref{thm:StabilityPhi}, we have proved that the density satisfies a discrete maximum principle which ensures, as it is noted in Remark~\ref{Rem:Bounds_Visc_Yield}, that the plastic viscosity and the yield stress have lower and upper bounds.
By using this preliminary stability result, we can now establish a stability result for all sequences computed with the scheme \eqref{PhaseEquation}-\eqref{UProj}.
\begin{theorem}
If $\boldsymbol{u}_0$ and $\rho_0$ satisfy \eqref{CondIniBoundModCont2}, if $ r \frac{\alpha_2^2}{\mu_1}\le\frac 32$ and $ \theta \leq \frac{1}{2} ,$ then, for all $ 1 \leq n \leq N$, we have
\begin{equation*}\begin{aligned}
\Vert \sqrt{\rho^n} \uu^{n} \Vert ^2 + \dt \sum _{k=0} ^{n-1} & \Vert \sqrt{2 \mu \enn} \DD \uu ^{k+1} \Vert ^2 +  \frac{\dt ^2}{\rho_1} \Vert\boldsymbol{\nabla}p^{n} \Vert ^2 + \frac{2 \theta}{r} \dt \Vert \Sig^{n} \Vert ^2  \\
& \leq \Vert \sqrt{\rho^0} \uu^{0} \Vert ^2 +  \frac{\dt ^2}{\rho_1} \Vert \boldsymbol{\nabla} p^{0} \Vert ^2 + \frac{2 \theta}{r} \dt \Vert \Sig^{0} \Vert ^2.
\end{aligned} 
\end{equation*}
\end{theorem}
\begin{proof}
We take the inner product of the first equation in \eqref{Momentum} with $ 2 \dt \uu^{n+1} $ in $\Ll2$. By observing that
\begin{equation*}
\Big ( 2 \uu^{n+1} , \frac{1}{2} ( \rho^{n+1} + \rho^n ) \uu^{n+1} - \rho^n \uu^n \Big ) = \Vert \sigma^{n+1} \uu^{n+1} \Vert ^2 - \Vert \sigma^{n} \uu^{n} \Vert ^2 + \Vert \sigma^{n} ( \uu^{n+1} - \uu^n) \Vert ^2 ,
\end{equation*}
where we denote $ \sigma^n = \sqrt{\rho^n} $, and recalling equation \eqref{CalculBaseStability}, we deduce
\begin{equation} \label{dem21}
\begin{aligned}
\Vert \sigma^{n+1} \uu^{n+1} \Vert ^2  &- \Vert \sigma^{n} \uu^{n} \Vert ^2 + \Vert \sigma^{n}( \uu^{n+1} - \uu^n) \Vert ^2 + 4\dt \Vert \sqrt{2 \mu^{n+1}} \DD \uu ^{n+1} \Vert ^2 \\
& = -2 \dt ( \boldsymbol{\nabla} (p^n + q^n ) , \uu^{n+1} ) - 2 \dt ( \alpha^{n+1} \Sig^{n+1} , \DD \uu^{n+1}).
\end{aligned}
\end{equation}
Now, we have to control the two terms in the right-hand side of \eqref{dem21} in order to obtain the expected inequality.\\

\textit{Term $-2 \dt ( \boldsymbol{\nabla} (p^n + q^n ) , \uu^{n+1} )$.} Recalling that $p \en - p ^{n-1} = q \en (\text{see } \eqref{IncPression} )$, we write
\begin{equation*} -2 \dt ( \boldsymbol{\nabla} (2p^n -p^{n-1} ) , \uu^{n+1} ) = 2 \dt ( \boldsymbol{\nabla} (p^{n+1} -2p^n +p^{n-1}) , \uu^{n+1} ) -2 \dt ( \boldsymbol{\nabla} p^{n+1} , \uu^{n+1} ).
\end{equation*}
Taking the inner product of the first equation in \eqref{Pression} with $\frac{2 \dt ^2}{\rho_1} \big ( (p^{n+1}-p^n)-(p^n-p^{n-1}) \big ) $ in $\L2$, we obtain
\begin{equation*}
\begin{aligned}
\frac{\dt ^2}{\rho_1} \Big ( \Vert \boldsymbol{\nabla} (p^{n+1}-p^n) \Vert ^2 - & \Vert \boldsymbol{\nabla} (p^{n}-p^{n-1}) \Vert ^2 + \Vert \boldsymbol{\nabla} (p^{n+1}-2p^n+p^{n-1}) \Vert ^2\Big ) \\
& = 2 \dt (\uu^{n+1} , \boldsymbol{\nabla} (p^{n+1}-2p^n+p^{n-1}) ).
\end{aligned}
\end{equation*}
Taking the inner product of the first equation in \eqref{Pression} with $-\frac{2 \dt ^2}{\rho_1} p^{n+1}$ in $\L2$, we obtain
\begin{equation*}
-\frac{\dt ^2}{\rho_1} \Big ( \Vert\boldsymbol{\nabla} p^{n+1} \Vert ^2 -  \Vert \boldsymbol{\nabla} p^{n} \Vert ^2 + \Vert \boldsymbol{\nabla} (p^{n+1}- p^n) \Vert ^2\Big ) = -2 \dt (\uu^{n+1} , \boldsymbol{\nabla} p^{n+1}).
\end{equation*}
Using the last three equalities, we deduce
\begin{equation} \label{dem22}
\begin{aligned}
-2 \dt ( \boldsymbol{\nabla} (2p^n -p^{n-1} )& , \uu^{n+1} ) =  \frac{\dt ^2 }{\rho_1} \Vert \boldsymbol{\nabla} (p^{n+1}-2p^n+p^{n-1}) \Vert ^2 \\
& +\frac{\dt ^2}{\rho_1} \Big [ - \Vert \boldsymbol{\nabla} p^{n+1} \Vert ^2 + \Vert \boldsymbol{\nabla} p^{n} \Vert ^2 - \Vert \boldsymbol{\nabla} (p^n-p^{n-1}) \Vert ^2\Big ]. 
\end{aligned}
\end{equation}
Using \eqref{Pression}, we have
\begin{equation} \label{dem23}
\frac{\dt^2}{\rho_1} \Vert \boldsymbol{\nabla} ( p^{n+1}-2p^n+p^{n-1}) \Vert ^2 \leq \Vert \sigma^n (\uu^{n+1}-\uu^n ) \Vert ^2 .
\end{equation}
Combining \eqref{dem21}, \eqref{dem22} and \eqref{dem23}, we finally deduce
\begin{equation} \label{dem24}
\begin{aligned}
  \Vert \sigma^{n+1} \uu^{n+1} \Vert ^2 & - \Vert \sigma^{n} \uu^{n} \Vert ^2 + 4 \dt \Vert \sqrt{ 2 \mu^{n+1}} \DD \uu ^{n+1} \Vert ^2 \\
  &  + \frac{\dt ^2}{\rho_1} \Big [ \Vert \boldsymbol{\nabla}p^{n+1} \Vert ^2 - \Vert \boldsymbol{\nabla} p^{n} \Vert ^2 + \Vert \boldsymbol{\nabla} (p^n-p^{n-1}) \Vert ^2\Big ] \\
  & \leq - 2 \dt ( \alpha^{n+1} \Sig^{n+1} , \DD \uu^{n+1}).
\end{aligned}
\end{equation}

\textit{Term $-2\dt ( \alpha^{n+1} \Sig^{n+1} , \DD \uu^{n+1})$.} Using the second equation of \eqref{Momentum}, we have
\begin{equation*}
\vert \Sig^{n+1} \vert \leq \big \vert \Sig^{n+1} + r \alpha^{n+1} \DD \uu ^{n+1} + \theta (\Sig^{n} - \Sig^{n+1}) \big \vert.
\end{equation*}
By taking the square of this inequality, using the identity $2a(a-b)=a^2-b^2+(a-b)^2,$ the Cauchy-Schwarz and Young's inequalities and integrating over $\Omega$, we obtain
\begin{equation} \label{dem25}
\begin{aligned}
  & - 2 \dt ( \alpha^{n+1} \Sig^{n+1} , \DD \uu ^{n+1} ) \leq  2 r   \frac{\alpha_2 ^2}{ \mu_1} \dt \Vert \sqrt{2 \mu^{n+1}} \DD \uu ^{n+1} \Vert ^2  \\
  &  \qquad\qquad+ \frac{2 \theta   (2 \theta - 1 )}{r} \dt \Vert \Sig^{n} - \Sig^{n+1}   \Vert ^2 - \frac{ 2 \theta  }{r} \dt \Vert \Sig^{n+1} \Vert ^2 + \frac{ 2 \theta  }{r} \dt \Vert \Sig^{n} \Vert ^2.
\end{aligned}
\end{equation}
Combining \eqref{dem24} and \eqref{dem25}, we obtain
\begin{equation*}\begin{aligned}
& \Vert \sigma^{n+1} \uu^{n+1} \Vert ^2  - \Vert \sigma^{n} \uu^{n} \Vert ^2 + 2\dt \Big ( 2 - r\frac{\alpha_2^2}{\mu_1} \Big ) \Vert \sqrt{2 \mu^{n+1}} \DD \uu ^{n+1} \Vert ^2 \\
& + \frac{\dt ^2}{\rho_1} \Big [ \Vert \boldsymbol{\nabla} p^{n+1} \Vert ^2 - \Vert \boldsymbol{\nabla} p^{n} \Vert ^2 + \Vert \boldsymbol{\nabla} (p^n-p^{n-1}) \Vert ^2\Big ] \\
& + \frac{2 \theta (1 - 2 \theta )}{r} \dt \Vert \Sig^{n} - \Sig^{n+1}   \Vert ^2 + \frac{2 \theta}{r} \dt \Vert \Sig^{n+1} \Vert ^2 - \frac{2 \theta}{r}\dt \Vert \Sig^{n} \Vert ^2 \leq 0.
\end{aligned} 
\end{equation*}
We conclude the proof by using $ r\frac{\alpha_2^2}{\mu_1} \le \frac 32$ and $ \theta \leq \frac 12 ,$ and summing from $0$ up to $n-1$ the last inequality.
\end{proof}

Note that the above estimate does not provide a bound on the pressure independent on $\dt$. We will derive such a stability result by using the final error estimate~\eqref{eqFinal} in the proof of Theorem~\ref{thm:ErreurFinal} (see Remark~\ref{rem:PressureBound}).\\
Moreover, using the second equation in \eqref{Momentum}, $ \Sig \en \in \mathbf{\Lambda} $. Hence, for all $ 0 \leq n \leq N$, we have
$$ \Vert \Sig \en \Vert _{\Linf} \leq 1. $$
\section{Error analysis}
As usual in error estimates, we have to assume that the solutions of \eqref{ModelContinuFinal} have a sufficient regularity. We assume in this section that the solution
$( \uu, \rho, p, \Sigma )$ of \eqref{ModelContinuFinal} satisfies
\begin{equation} \label{Regularity_Assumptions}
\begin{aligned}
  & \uu \in  L^{\infty}(0,T;\mathbf{L}^{\infty}(\Omega)), \quad \boldsymbol{\nabla} \uu \in  L^{\infty}(0,T; \mathbf{L}^{\infty}(\Omega)), \quad \partial _t \uu \in  L^{\infty}(0,T; \mathbf{L} ^{3}(\Omega)), \\
  & \partial _{tt} \uu \in  L^{2}(0,T; \mathbf{L}^{6/5}(\Omega)), \quad \boldsymbol{\nabla} \rho \in  L^{\infty}(0,T; \mathbf{L} ^{\infty}(\Omega)), \\
  & \partial_{tt} \rho \in  L^{2}(0,T;L^{2}(\Omega)), \quad \partial _t p \in  L^{2}(0,T,H^1(\Omega)), \quad \partial _t \Sig \in  L^{2}(0,T; \mathbf{L} ^{2}(\Omega)). 
\end{aligned}
\end{equation}
Under this regularity assumptions on $\uu$ (see~\cite{boyerfabrie}), the solution $\rho$ of the mass conservation equation \eqref{ModelContinuEq1} satisfies the maximum principle, that is almost everywhere in $(0,T)\times \Omega$, we have
$$ \rho _1 \leq \rho \leq \rho _2. $$
We introduce the following errors
\begin{align*}
& \varepsilon^n = \rho(t_n) - \rho^n, \quad  \ee^n = \uu(t_n) - \uu^n, \quad \hat{\ee} \en = \uu ( \tn) - \hat{\uu} \en , \\
& r^n = p(t_n) - p^n, \quad \sss ^n = \Sig(t_n) - \Sig ^n.
\end{align*}
We denote $\Vert\mu^\prime\Vert_\infty=sup _{x \in \mathbb{R}} \vert \mu^\prime(x)\vert$ and $\Vert \alpha^\prime\Vert_\infty = \sup _{x \in \mathbb{R}} \vert \alpha^\prime(x) \vert$. Since $\mu$ and $\alpha$ are functions of class $\mathcal{C}^1$, we obtain
$$ \Vert \mu ( \rho ( \tn ) ) - \mu \en \Vert \leq \Vert \mu '  \Vert _\infty \Vert \varepsilon \en \Vert \quad \text{and} \quad \Vert \alpha ( \rho ( \tn ) ) - \alpha \en \Vert \leq \Vert \alpha '  \Vert _\infty \Vert \varepsilon \en \Vert.$$
In the rest of the paper, $C$ denotes a constant independent of the time step $\dt$.
Since $p|_{t=0} $ is not prescribed in the initial conditions, we can not choose $p^0 = p|_{t=0}. $ Therefore $ r^0 \ne 0 $ so that we assume
\begin{equation}
\label{HypErrPressIni}
\Vert r^0 \Vert ^2 + \Vert \boldsymbol{\nabla} r^0 \Vert ^2 \leq C \dt ^2.
\end{equation}
We will also assume
\begin{equation}\label{AssumptionDt}
\dt < 1/6.
\end{equation}
\begin{rem} \label{EEchap}
Due to \eqref{UProj} and the above definitions, we have $ \ee \en = \hat{\ee} \en - \frac{\dt}{\rho_1} \boldsymbol{\nabla} q \en.$
Recalling that $\Div \eee \en =0 $ and $ \eee \en \cdot \nn |_{\Gamma} =0$ (see Remark~\ref{RemUChap}), we deduce
\begin{equation*}
\Vert \hat{\ee} \en \Vert ^2 \leq  \Vert \hat{\ee} \en \Vert ^2 + \frac{\dt ^2 }{\rho_1 ^2 } \Vert  \boldsymbol{\nabla} q \en \Vert ^2 =\Vert \ee \en \Vert ^2 .
\end{equation*}
\end{rem}

\subsection{Preliminary error estimates for the density}
As for the stability analysis, we first focus on the sequence $(\rho^n)_{\lbrace 0\leq n\leq N \rbrace}$ governed by equation \eqref{PhaseEquation}.
\begin{theorem} \label{thm:ErreurPhi}
Under the regularity assumptions \eqref{Regularity_Assumptions}, if $\boldsymbol{u}_0$ and $\rho_0$ satisfy \eqref{CondIniBoundModCont2}, then, for all $n$ such that $0 \leq n \leq N-1$, we have
\begin{equation} \label{ThmErreurPhiEq1}
\dt \sum_{k=0}^{n} \Vert  \rho ( t_{k+1} ) - \rho^{k+1} \Vert ^2 \leq C \dt ^2 + C \dt \sum_{k=0}^{n} \Vert \uu(t_{k+1})  - \uu^{k+1} \Vert ^2.
\end{equation}
\end{theorem}
\begin{proof}
The $\Linf $-regularity of $\rho\enn$ is not sufficient to derive the expected result. For this reason, the regularized solution $\rho\en _\epsilon \in \Hach1$, solution of~\eqref{Viscosity}, is considered. We denote $\varepsilon \en _{\epsilon} = \rho(\tn) - \rho\en _{\epsilon}$. 
By subtracting the first equation of \eqref{ModelContinuEq1} taken at time $t_{n+1}$ from \eqref{Viscosity}, we obtain
\begin{equation}\label{ErrorEqEps}
  \begin{aligned}
\frac{\varepsilon^{n+1} _{\epsilon} -\varepsilon^n}{\dt} & = \frac{\rho(\tnn)-\rho(\tn)}{\dt}- \partial_t \rho(\tnn) - \hat{\uu}^n \cdot \boldsymbol{\nabla} \varepsilon \enn _{\epsilon} \\
& - \hat{\ee}^n \cdot \boldsymbol{\nabla} \rho (\tnn) - \delta \uu (\tnn) \cdot \boldsymbol{\nabla} \rho(\tnn) - \epsilon \Delta \rho\enn _{\epsilon},
\end{aligned} 
\end{equation}
where by invoking Taylor's formulae, we have
\begin{equation*}
 \frac{\rho(\tnn)-\rho(\tn)}{\dt}- \partial_t \rho(\tnn) = \frac{1}{\dt} \int_{t_n}^{\tnn} (\tn - t) \partial_{tt}\rho(t)dt.
\end{equation*}
Now, by taking the $\L2$-inner product of the error equation \eqref{ErrorEqEps} with $2 \dt \varepsilon\enn_{\epsilon} $ and using the above equality, we obtain
\begin{equation*} \begin{aligned}
&  \Vert \varepsilon \enn _{\epsilon} \Vert ^2 - \Vert \varepsilon \en \Vert ^2 + \Vert \varepsilon \enn _{\epsilon} -\varepsilon \en \Vert ^2
                                                         = 2 \bigg ( \int_{t_n}^{\tnn} (\tn - t) \partial_{tt}\rho(t)dt ,  \varepsilon _{\epsilon} \enn \bigg ) \\
&\quad - 2 \dt \int_{\Omega}  \varepsilon \enn _{\epsilon} \bigl(\eee ^n \cdot \boldsymbol{\nabla} \rho(\tnn)+ \delta \uu ( \tnn) \cdot \boldsymbol{\nabla}\rho( \tnn) \bigr) 
                                              + 2 \dt \epsilon (\boldsymbol{\nabla} \rho \enn _{\epsilon}, \boldsymbol{\nabla} \varepsilon \enn _{\epsilon} ). 
\end{aligned} 
\end{equation*}
By writing
\begin{equation*}
  2 \dt \epsilon (\boldsymbol{\nabla}\rho \enn _{\epsilon}, \boldsymbol{\nabla} \varepsilon \enn _{\epsilon} ) = 2 \dt \epsilon (\boldsymbol{\nabla} \rho \enn _{\epsilon}, \boldsymbol{\nabla} \rho ( \tnn)  ) - 2 \dt \Vert \sqrt{\epsilon}\boldsymbol{\nabla} \rho \enn _{\epsilon} \Vert^2,
\end{equation*}
we obtain
\begin{equation} \label{avlimite}
\begin{aligned}
\Vert \varepsilon \enn _{\epsilon} \Vert ^2 & - \Vert \varepsilon \en \Vert ^2 + \Vert \varepsilon \enn _{\epsilon} -\varepsilon \en \Vert ^2 + 2 \dt \Vert \sqrt{\epsilon} \boldsymbol{\nabla} \rho \enn _{\epsilon} \Vert ^2  \\
& = 2 \bigg ( \int_{t_n}^{\tnn} (\tn - t) \partial_{tt}\rho(t)dt ,  \varepsilon _{\epsilon} \enn \bigg ) - 2 \dt \int_{\Omega}  \eee ^n \cdot \boldsymbol{\nabla} \rho (\tnn) \varepsilon \enn _{\epsilon} \\
& - 2 \dt \int_{\Omega}  \delta \uu ( \tnn) \cdot \boldsymbol{\nabla} \rho ( \tnn) \varepsilon \enn _{\epsilon} + 2 \dt \epsilon (\boldsymbol{\nabla} \rho \enn _{\epsilon}, \boldsymbol{\nabla} \rho ( \tnn) ). 
\end{aligned}
\end{equation}
Making $\epsilon$ tends to $0$ in \eqref{avlimite}, using \eqref{CVFBrezis} and the inequality $\Vert \varepsilon \enn \Vert \leq  \lim \inf _{\epsilon} \Vert \varepsilon \enn _{\epsilon} \Vert,$ we derive
\begin{equation} \label{dem42}
\begin{aligned}
  \Vert \varepsilon \enn \Vert ^2 & - \Vert \varepsilon \en \Vert ^2 \leq 2 \bigg ( \int_{t_n}^{\tnn} (\tn - t) \partial_{tt}\rho(t)dt ,  \varepsilon \enn \bigg ) \\
   & - 2 \dt \int_{\Omega}  \eee ^n \cdot \boldsymbol{\nabla} \rho (\tnn) \varepsilon \enn 
     - 2 \dt \int_{\Omega}  \delta \uu ( \tnn) \cdot \boldsymbol{\nabla} \rho ( \tnn) \varepsilon \enn . 
\end{aligned}
\end{equation}
We now estimate the terms in the right-hand side of \eqref{dem42}
\begin{equation*} \begin{aligned} 
    2 \int_{\Omega} \Bigl( \int_{t_n}^{\tnn} (\tn - t) \partial_{tt}dt \rho(t)dt\Bigr)\, \varepsilon\enn
                                    & \leq \dt ^2 \int_{t_n}^{\tnn} \Vert \partial_{tt} \rho( t) \Vert^2 dt + \dt \Vert \varepsilon  \enn \Vert ^2, \\
    - 2 \dt \int_{\Omega}  \eee ^n \cdot \boldsymbol{\nabla} \rho (\tnn) \varepsilon \enn  & \leq \dt \Vert \boldsymbol{\nabla} \rho  \Vert _{L^{\infty}(0,T,L ^{\infty}(\Omega))} ^2 \Vert \eee ^n \Vert ^2 + \dt  \Vert \varepsilon \enn \Vert ^2,
\end{aligned}
\end{equation*}
and
\begin{equation*} \begin{aligned} 
   & - 2 \dt \int_{\Omega}  \delta \uu ( \tnn) \cdot \boldsymbol{\nabla} \rho ( \tnn) \varepsilon \enn \\
   & \phantom{- 2 \dt \int_{\Omega}  \delta \uu ( \tnn) \cdot }   \leq 2  \dt ^\frac{3}{2} \Big ( \int_{\tn} ^{\tnn}  \Vert \partial _t \uu (t) \Vert _{L ^{4}(\Omega)} ^2 dt \Big ) ^\frac{1}{2} \Vert \boldsymbol{\nabla} \rho  \Vert _{L^{\infty}(0,T,L ^{4}(\Omega))} \Vert \varepsilon \enn \Vert \\
  & \phantom{- 2 \dt \int_{\Omega}  \delta \uu ( \tnn) \cdot } \leq \dt ^2 \Vert \boldsymbol{\nabla} \rho  \Vert ^2 _{L^{\infty}(0,T,L ^{4}(\Omega))} \int_{\tn} ^{\tnn}  \Vert \partial _t \uu (t) \Vert _{L ^{4}(\Omega)} ^2  dt  + \dt \Vert \varepsilon \enn \Vert ^2.
\end{aligned}\end{equation*}
Reporting the above inequalities in \eqref{dem42} and summing over $n$ from $0$ to $j$, with $ 0 \leq j \leq N-1$, we have
\begin{equation*}
\begin{aligned}
\Vert \varepsilon ^{j+1} \Vert ^2 & \leq  \dt ^2 \big ( \Vert \partial_{tt} \rho  \Vert ^2 _{L^{2}(0,T,L ^{2}(\Omega))} + \Vert \partial _t \uu \Vert ^2 _{L^{2}(0,T,L ^{4}(\Omega))} \Vert \boldsymbol{\nabla} \rho  \Vert ^2 _{L^{\infty}(0,T,L ^{4}(\Omega))} \big ) \\
& +3 \dt \sum_{n=0} ^j \Vert \varepsilon \enn \Vert ^2 + \dt \Vert \boldsymbol{\nabla} \rho \Vert _{L^{\infty}(0,T,L ^{\infty}(\Omega))} ^2 \sum _{n=0}^{j} \Vert \eee ^n \Vert ^2 .
\end{aligned}
\end{equation*}
Using the discrete Gronwall Lemma~\ref{lem:LGdiscret}, the assumption~\eqref{AssumptionDt}, and Remark~\ref{EEchap} we derive
\begin{equation*}
\begin{aligned}
\Vert \varepsilon ^{j+1} \Vert ^2 \leq  &  \dt ^2 \exp (6T) \Big ( \Vert \partial_{tt} \rho  \Vert ^2 _{L^{2}(0,T,L ^{2}(\Omega))} + \Vert \partial _t \uu \Vert ^2 _{L^{2}(0,T,L ^{4}(\Omega))} \Vert \boldsymbol{\nabla} \rho  \Vert ^2 _{L^{\infty}(0,T,L ^{4}(\Omega))} \Big ) \\
& + \dt \exp (6T) \Vert \boldsymbol{\nabla} \rho  \Vert _{L^{\infty}(0,T,L ^{\infty}(\Omega))} ^2 \sum _{n=0}^{j} \Vert \ee ^n \Vert ^2 .
\end{aligned}
\end{equation*}
We conclude the proof of Theorem~\ref{thm:ErreurPhi} by multiplying by $\dt$ and summing over $j$ from~$0$ to~$k$, with $k\leq N-1$, the above inequality.
\end{proof}

\subsection{Final error estimates}
\begin{theorem} \label{thm:ErreurFinal}
Under the regularity assumption \eqref{Regularity_Assumptions}, if the initial conditions satisfy~\eqref{CondIniBoundModCont2} and~\eqref{HypErrPressIni}, 
if $ \theta < \frac{1}{2},\, \delta t<1/6,\, r\alpha_2^2 \leq \frac{\mu_1}{12 (1 + C(\Omega)^2)} $ where $C(\Omega)$ is the Poincar\'e's constant,
then, for all $n$ such that $ 1 \leq n \leq N $, the discrete velocity solution of the scheme \eqref{PhaseEquation}-\eqref{UProj} satisfies the
following error estimate
\begin{equation*} 
\begin{aligned}
&\Vert  \uu(t_n) - \uu ^n \Vert ^2 + \dt \sum_{k=0} ^{N-1} \Vert \uu(t_{k+1}) - \uu ^{k+1} \Vert ^2 _1 \leq C \dt ( \dt + \theta ).
\end{aligned}
\end{equation*}
\end{theorem}

\begin{proof}
We first derive the error equation which is obtained by subtracting the two first equations in \eqref{ModelContinuFinal} to the first equation in \eqref{Momentum}, namely we find
\begin{equation} \label{ErrorEq0}
\begin{aligned}
  & \frac{1}{\dt} \Big [ \frac{1}{2} ( \rho^{n+1} + \rho \en ) \ee \enn - \rho \en \ee \en \Big ] - \Div \big (2 \mu \enn \DD \ee \enn \big ) \\
  & \quad + \boldsymbol{\nabla} \big ( p(\tnn) - (p \en + q \en ) \big ) = \Div \big ( \alpha( \rho (\tnn) ) \Sig (\tnn) - \alpha \enn \Sig \enn \big ) \\
  & \quad + \Div \Big ( 2 \big (  \mu ( \rho (\tnn) ) - \mu \enn \big ) \DD \uu (\tnn) \Big ) + \boldsymbol{R}^{n+1}_1 + \boldsymbol{R}^{n+1}_2 + \boldsymbol{R}^{n+1}_3 + \boldsymbol{R}^{n+1}_4,
\end{aligned}
\end{equation}
where
\begin{equation*} \begin{aligned}
  & \boldsymbol{R}^{n+1}_1 = \rho \en \frac{\uu ( \tnn) - \uu ( \tn) }{\dt} -  \rho (\tnn) \partial_t \uu ( \tnn), \\
  & \boldsymbol{R}^{n+1}_2 = \frac{1}{2} \uu (\tnn) \bigg ( \frac{ \rho \enn - \rho \en }{\dt} - \partial _t  \rho (\tnn) \bigg ), \\
  & \boldsymbol{R}^{n+1}_3 = \rho \enn \uuu \en \cdot \boldsymbol{\nabla} \uu \enn - \rho (\tnn) \uu (\tnn) \cdot \boldsymbol{\nabla} \uu ( \tnn), \\
  & \boldsymbol{R}^{n+1}_4 = \frac{1}{2} \uu \enn \Div ( \rho \enn \uuu \en ) - \frac{1}{2} \uu ( \tnn ) \Div \big ( \uu (\tnn) \rho (\tnn) \big ).
\end{aligned} \end{equation*}
By virtue of the Korn's inequality (see \cite{Korn}), since $\ee \enn \in \Hh10 $, we have $\Vert \boldsymbol{\nabla} \ee \enn \Vert ^2 \leq 2 \Vert \DD \ee \enn \Vert ^2.$
Invoking the Poincar\'e's inequality ($\Vert \ee \enn \Vert \leq C(\Omega) \Vert \boldsymbol{\nabla} \ee \enn \Vert $) and denoting $ C_P = (1 + C(\Omega)^2)^{-1}$, we obtain
\begin{equation*}
  4 \mu_1 C_P \dt \Vert \ee \enn \Vert ^2 _1 \leq 4 \dt \Vert \sqrt{2 \mu \enn} \DD \ee \enn \Vert ^2.
\end{equation*}
We now take the inner product of equation \eqref{ErrorEq0} with $2\dt \ee \enn $ in $\Ll2$. Using the last inequality we obtain
\begin{equation} \label{ErrorIneq0}
  \begin{aligned}
&\Vert \sigma \enn \ee \enn \Vert ^2 + \Vert \sigma \en( \ee \enn - \ee \en ) \Vert ^2 - \Vert \sigma \en \ee \en \Vert ^2 + 4 \mu_1 C_P \dt \Vert \ee \enn \Vert ^2 _1 \\
  & \leq -2 \dt \Big (  \boldsymbol{\nabla} \big ( p(\tnn) - (p \en + q \en ) \big ) , \ee \enn \Big ) \\
  & \phantom{\leq} -2\dt \big ( \alpha(\rho (\tnn)) \Sig (\tnn) - \alpha \enn \Sig \enn , \DD \ee \enn \big )  \\
  & \phantom{\leq} -4 \dt \Big ( \big (  \mu ( \rho (\tnn) ) - \mu \enn \big ) \DD \uu ( \tnn ) , \DD \ee \enn \Big ) + 2 \dt \Big ( \boldsymbol{R}^{n+1}_1 , \ee \enn \Big ) \\
  & \phantom{\leq} + 2 \dt \Big ( \boldsymbol{R}^{n+1}_2 , \ee \enn \Big ) + 2 \dt \Big ( \boldsymbol{R}^{n+1}_3 , \ee \enn \Big ) + 2 \dt \Big (\boldsymbol{R}^{n+1}_4 , \ee \enn \Big ).
  \end{aligned}
\end{equation}
In the following steps, we derive estimates of the terms occurring in the right-hand side of~\eqref{ErrorIneq0}. They are treated in order of appearance in~\eqref{ErrorIneq0}.\\

\textit{Step 1.} From \eqref{Pression} we obtain, for all $g\in H^1(\Omega)$,
\begin{equation} \label{ErreurPression}
\big ( \boldsymbol{\nabla} ( r \enn - r \en ) , \boldsymbol{\nabla} g \big ) = \frac{\rho_1}{\dt} ( \ee \enn , \boldsymbol{\nabla} g ) + ( \boldsymbol{\nabla} \delta p (\tnn) , \boldsymbol{\nabla} g ).
\end{equation}
Taking $g = - \frac{2 \dt ^2}{\rho_1} \delta^2 r \enn $ in \eqref{ErreurPression}, we have
\begin{align*}
  - \frac{\dt^2}{\rho_1} \Big ( \Vert \boldsymbol{\nabla} ( r \enn - r \en ) \Vert ^2 & - \Vert \boldsymbol{\nabla} ( r \en - r \enm ) \Vert ^2 + \Vert \boldsymbol{\nabla} \delta^2 r \enn  \Vert ^2 \Big ) \\
  & = -2 \dt ( \ee \enn , \boldsymbol{\nabla} \delta^2 r \enn ) - \frac{2 \dt ^2}{\rho_1} \big ( \boldsymbol{\nabla} \delta p ( \tnn ) , \boldsymbol{\nabla} \delta^2 r \enn \big ) .
\end{align*}
Taking $g = \frac{2 \dt ^2}{\rho_1}  r \enn $ in \eqref{ErreurPression}, we have
\begin{align*}
  \frac{\dt^2}{\rho_1} \Big ( \Vert \boldsymbol{\nabla} r \enn \Vert ^2 & - \Vert \boldsymbol{\nabla}  r \en \Vert ^2 + \Vert \boldsymbol{\nabla} \delta  r \enn  \Vert ^2 \Big ) \\
  & = 2 \dt ( \ee \enn , \boldsymbol{\nabla}  r \enn ) + \frac{2 \dt ^2}{\rho_1} \big (\boldsymbol{\nabla} \delta p ( \tnn ) , \boldsymbol{\nabla}  r \enn \big ) .
\end{align*}
By summing these last two equations, we deduce
\begin{equation} \label{dem62}
\begin{aligned}
  &\frac{\dt^2}{\rho_1} \Big ( \Vert \boldsymbol{\nabla} r \enn \Vert ^2 - \Vert \boldsymbol{\nabla}  r \en \Vert ^2 + \Vert \boldsymbol{\nabla} \delta  r \en  \Vert ^2 \Big ) - \frac{\dt ^2}{\rho_1} \Vert \boldsymbol{\nabla} \delta^2 r \enn \Vert ^2  \\
  & \qquad = 2 \dt \big ( \ee \enn , \boldsymbol{\nabla}  (2 r \en - r \enm) \big ) + \frac{2 \dt ^2}{\rho_1} \big ( \boldsymbol{\nabla} \delta p ( \tnn ) , \boldsymbol{\nabla}  (2r \en - r \enm ) \big ) .
\end{aligned}
\end{equation}
From \eqref{ErreurPression}, we have
\begin{equation*}
\big ( \boldsymbol{\nabla} \delta^2 r \enn , \boldsymbol{\nabla} g \big ) = \frac{\rho_1}{\dt} ( \delta \ee \enn , \boldsymbol{\nabla} g ) + ( \boldsymbol{\nabla} \delta^2 p (\tnn) , \boldsymbol{\nabla} g ).
\end{equation*}
Taking $g = \frac{\dt}{\sqrt{\rho_1}} \delta^2 r \enn $ in the above relation, we obtain
\begin{equation} \label{dem63}
\begin{aligned}
&\frac{\dt^2}{\rho_1} \Vert \boldsymbol{\nabla} \delta^2 r \enn \Vert ^2  \leq \frac{1}{\rho_1} \Vert \rho_1  \delta \ee \enn + \dt \boldsymbol{\nabla} \delta^2 p ( \tnn) \Vert ^2 \\
&\qquad\qquad \leq \Vert \sigma ^n \delta \ee \enn \Vert ^2 + \frac{\dt ^2}{\rho_1} \Vert \boldsymbol{\nabla} \delta^2 p ( \tnn ) \Vert ^2 + 2 \dt \big ( \delta \ee \enn , \boldsymbol{\nabla} \delta^2 p(\tnn) \big ).  
\end{aligned} 
\end{equation}
By summing \eqref{dem62} and \eqref{dem63}, we deduce
\begin{equation*}
\begin{aligned}
 \frac{\dt^2}{\rho_1} \Big ( \Vert \boldsymbol{\nabla} r \enn \Vert ^2 - & \Vert \boldsymbol{\nabla}  r \en \Vert ^2 + \Vert \boldsymbol{\nabla} \delta  r \en  \Vert ^2 \Big ) \leq  2 \dt \big ( \ee \enn , \boldsymbol{\nabla}  (2 r \en - r \enm) \big ) \\
 & + \frac{2 \dt ^2}{\rho_1} \big ( \delta p ( \tnn ) , \boldsymbol{\nabla}  (2r \en - r \enm ) \big ) + \Vert \sigma ^n \delta \ee \enn \Vert ^2 \\
 &  + \frac{\dt ^2}{\rho_1} \Vert \boldsymbol{\nabla} \delta^2 p ( \tnn ) \Vert ^2 + 2 \dt \big ( \delta \ee \enn , \boldsymbol{\nabla} \delta^2 p(\tnn) \big ) ,
\end{aligned}
\end{equation*}
which can be rewritten as
\begin{equation*}
\begin{aligned}
 - & 2 \dt \big ( \ee \enn , \boldsymbol{\nabla}  (2 r \en - r \enm) \big ) - 2 \dt \big ( \delta \ee \enn , \boldsymbol{\nabla} \delta^2 p(\tnn) \big ) -2 \dt \big ( \ee \en , \boldsymbol{\nabla} \delta^2 p ( \tnn) \big ) \\
\leq & \underbrace{-2 \dt \big ( \ee \en , \boldsymbol{\nabla} \delta^2 p ( \tnn) \big )}_{a_1} - \frac{\dt^2}{\rho_1} \Big ( \Vert \boldsymbol{\nabla} r \enn \Vert ^2 -  \Vert \boldsymbol{\nabla}  r \en \Vert ^2 + \Vert \boldsymbol{\nabla} \delta  r \en  \Vert ^2 \Big )  \\
& \underbrace{+ \frac{2 \dt ^2}{\rho_1} \big ( \delta p ( \tnn ) , \boldsymbol{\nabla}  (2r \en - r \enm ) \big )}_{a_2} + \Vert \sigma ^n \delta \ee \enn \Vert ^2 + \underbrace{ \frac{\dt ^2}{\rho_1} \Vert \boldsymbol{\nabla} \delta^2 p ( \tnn ) \Vert ^2}_{a_3}.
\end{aligned}
\end{equation*}
Using the Cauchy-Schwarz and Young's inequalities, we obtain
\begin{equation*} \begin{aligned}
a_1  &\leq \dt \Vert \ee \en \Vert ^2 +  2 \dt ^2 \int _{\tnm} ^{\tnn} \Vert \partial _t p (s) \Vert_1 ^2 ds, \\
a_2  & = \frac{2 \dt ^2}{\rho_1} \Big ( \big ( \delta p ( \tnn ) , \boldsymbol{\nabla} r \en \big ) + \big ( \delta p ( \tnn ) , \boldsymbol{\nabla}  \delta r \en  \big ) \Big ) \\ 
     & \leq \frac{2 \dt ^2}{\rho_1} \int_{t_n} ^{\tnn}  \Vert \partial _t p (s) \Vert ^2 ds + \frac{\dt ^3}{\rho_1} \Vert \boldsymbol{\nabla} r \en \Vert ^2 + \frac{\dt ^3}{\rho_1} \Vert \boldsymbol{\nabla} \delta r \en \Vert ^2, \\
a_3  & \leq \frac{2 \dt^3}{\rho_1} \int _{\tnm} ^{\tnn} \Vert \partial _t p(s) \Vert ^2 _1 ds.
\end{aligned} 
\end{equation*}
We finally deduce
\begin{equation*}  \begin{aligned}
 -2 \dt & \Big (  \boldsymbol{\nabla} \big ( p(\tnn) - (p \en + q \en ) \big ) , \ee \enn \Big ) \leq 
- \frac{\dt^2}{\rho_1} \Big ( \Vert \boldsymbol{\nabla} r \enn \Vert ^2 -  \Vert \boldsymbol{\nabla}  r \en \Vert ^2 \Big ) + \frac{\dt ^3}{\rho_1} \Vert \boldsymbol{\nabla} r \en \Vert ^2  \\
& \qquad + \Vert \sigma ^n \delta \ee \enn \Vert ^2 + \dt ^2 \bigg ( \frac{2+4 \dt}{\rho_1} +1 \bigg ) \int_{\tnm} ^{\tnn} \Vert \partial _t p(s) \Vert ^2 _1 ds + \dt \Vert \ee \en \Vert ^2 .
\end{aligned} 
\end{equation*}
\textit{Step 2.} We write
\begin{equation} \label{Dem6Etape1}
\begin{aligned}
-2\dt \big (& \alpha(\rho (\tnn))  \Sig (\tnn) - \alpha \enn \Sig \enn , \DD \ee \enn \big )  \\
& = -2\dt \Big ( \big ( \alpha ( \rho ( \tnn) ) - \alpha \enn \big )  \Sig (\tnn) , \DD \ee \enn \Big ) - 2 \dt \big ( \sss \enn , \alpha \enn \DD \ee \enn \big )\\
& \leq \dt \frac{ \mu_1 C_P}{2} \Vert \ee \enn \Vert _1 ^2 + \dt  \frac{2 \Vert \alpha ' \Vert _\infty ^2}{\mu_1 C_P}  \Vert \varepsilon \enn \Vert ^2 - 2 \dt \big ( \sss \enn , \alpha \enn \DD \ee \enn \big ).
\end{aligned}
\end{equation}
In order to treat the last term in the right-hand side of \eqref{Dem6Etape1}, we use the relation \eqref{PartiePlastiqueContinueD2} on $\Sig$ with $\ell = r \alpha \enn $, namely 
\begin{equation*}
  \Sig ( \tnn ) = \mathbb{P} \big ( \Sig ( \tnn) + r \alpha \enn  \DD \uu ( \tnn) \big ),
\end{equation*}
so that
\begin{equation*}
\sss \enn = \mathbb{P} \big ( \Sig ( \tnn ) + r \alpha \enn  \DD \uu ( \tnn) \big ) - \mathbb{P} \big ( \Sig^{n+1} + r \alpha^{n+1} \DD \uu ^{n+1} + \theta (\Sig^{n} - \Sig^{n+1}) \big ).
\end{equation*}
Since $\mathbb{P}$ is a projection, we have
\begin{equation*}
\vert \sss \enn \vert \leq \vert \sss \enn + r \alpha \enn \DD \ee \enn + \theta ( \sss \en - \sss \enn ) + \theta \delta \Sig ( \tnn ) \vert .
\end{equation*}
Taking the square of this inequality leads to
\begin{equation*}
\begin{aligned}
\vert \sss \enn \vert ^2  \leq & \vert \sss \enn \vert ^2 + r^2 (\alpha \enn )^2 \vert \DD \ee \enn \vert ^2 + \theta ^2 \vert \sss \en - \sss \enn \vert ^2 + \theta ^2 \vert \delta \Sig ( \tnn ) \vert ^2 \\
& +  r \sss \enn :  \alpha \enn \DD \ee \enn +  \theta \sss \enn : (\sss \en - \sss \enn ) +  \theta \sss \enn : \delta \Sig ( \tnn ) \\
& +  r \theta \alpha \enn \DD \ee \enn : ( \sss \en - \sss \enn ) +   r \theta \alpha \enn \DD \ee \enn :  \delta \Sig ( \tnn )  \\
& +  \theta ^2 ( \sss \en - \sss \enn ) :  \delta \Sig ( \tnn )  .
\end{aligned}
\end{equation*}
Using the Cauchy-Schwarz and the Young's inequalities, the identity $2a(a-b)=a^2- b^2 + (a-b)^2,$ integrating over $\Omega$, and multiplying by $\frac{2 \dt}{r}$, we obtain
\begin{equation*}
\begin{aligned}
 - 2 \dt ( \sss \enn , \alpha \enn \DD \ee \enn ) & \leq   6 \dt r \alpha _2 ^2 \Vert \DD \ee \enn \Vert ^2 + \frac{2 \theta }{r} (3 \theta \dt + 1 ) \Vert \delta \Sig ( \tnn ) \Vert ^2  \\
& - \frac{2 \theta}{r} \dt ( 1 - \dt ) \Vert \sss \enn \Vert ^2 + \frac{2 \theta}{r} \dt \Vert \sss \en \Vert ^2 .
\end{aligned}
\end{equation*}
By reporting the last inequality in \eqref{Dem6Etape1} we deduce
\begin{equation*}\begin{aligned}
& -2 \dt \big ( \alpha(\rho (\tnn) ) \Sig (\tnn) - \alpha \enn \Sig \enn , \DD \ee \enn \big )  \\
& \qquad\leq \Big ( \frac{\mu_1 C_P}{2} + 6  r \alpha_2 ^2 \Big ) \dt  \Vert \ee \enn \Vert _1 ^2 + \dt \frac{2 \Vert \alpha ' \Vert _\infty ^2}{ \mu_1 C_P} \Vert \varepsilon \enn \Vert ^2 \\
& \qquad\qquad+ \dt \frac{ 2 \theta(3 \theta \dt +1)}{r} \int_{\tn} ^{\tnn} \Vert \partial_t \Sig (s) \Vert ^2 ds - \frac{2 \dt \theta}{r} \Big ( (1 - \dt ) \Vert \sss \enn \Vert ^2 - \Vert \sss \en \Vert ^2 \Big ).
\end{aligned} 
\end{equation*}
\textit{Step 3.} The third term in the right-hand side of \eqref{ErrorIneq0} is bounded as follows
\begin{equation*}\begin{aligned}
- 4 \dt \Big ( & \big (  \mu ( \rho (\tnn) ) - \mu \enn \big ) \DD \uu ( \tnn ) , \DD \ee \enn \Big ) \\
\leq & \dt \frac{ 8 \Vert \mu ' \Vert _\infty ^2}{\mu_1 C_P} \Vert \DD \uu \Vert_{L^{\infty}(0,T,L^{\infty}(\Omega))} ^2  \Vert \varepsilon \enn \Vert ^2 + \dt \frac{\mu_1 C_P}{2} \Vert \ee \enn \Vert_1^2.
\end{aligned} 
\end{equation*}
\textit{Step 4.} Using the Sobolev embedding $ \Haach1 \subset \LL ^6(\Omega) $ in addition to the Cauchy-Schwarz inequality, the Young's inequality and Taylor's formulae, we deduce
\begin{equation*}
 \begin{aligned}
& 2 \dt (\boldsymbol{R}^{n+1}_1 , \ee \enn) \leq   \frac{6 C_{S}}{\mu_1 C_P} \dt \Vert \varepsilon ^n \Vert ^2 \Vert \partial_t \uu \Vert _{L^{\infty} (0,T,L^3(\Omega))} ^2 + \dt \frac{\mu_1 C_P }{2 } \Vert \ee \enn \Vert ^2 _{1} \\
&  +\frac{6 C_{S}}{\mu_1 C_P} \dt ^2  \Big ( \rho^2_2 \int_{t_n} ^{\tnn} \Vert \partial _{tt} \uu (s) \Vert ^2 _{L^{6/5}(\Omega)} ds + \Vert \partial _{t} \uu \Vert ^2 _{L^{\infty} (0,T,L^3(\Omega))} \int_{t_n} ^{\tnn} \Vert \partial _{t} \rho (s) \Vert ^2 ds \Big ), 
\end{aligned} 
\end{equation*}
where $C_{S}$ is the continuity constant from the Sobolev embedding $ \mathbf{H}^1(\Omega) \subset \mathbf{L}^6(\Omega)$, that is, for all $f\in \mathbf{H}^1(\Omega),$ we have $\Vert f\Vert_{L^6(\Omega)}^2 \leq C_S \Vert f\Vert_1^2.$ \\

\textit{Step 5.} Using \eqref{PhaseEquation} and the first equation of \eqref{ModelContinuFinal}, we write
\begin{equation*}
  \boldsymbol{R}^{n+1}_2 = \frac{1}{2} \uu(\tnn) \bigg( \delta \uu(\tnn)\cdot\boldsymbol{\nabla}\rho(\tnn) + \uu(\tn)\cdot\boldsymbol{\nabla}\varepsilon \enn + \eee \en\cdot\boldsymbol{\nabla} \rho \enn  \bigg ).
\end{equation*}
Recalling that both $\uu$ and $\eee\en$ are divergence-free vector fields, we obtain
\begin{equation*}
  \boldsymbol{R}^{n+1}_2 = \frac{1}{2} \uu ( \tnn ) \Div \bigg ( \delta \uu ( \tnn ) \rho ( \tnn ) + \uu ( \tn) \varepsilon \enn + \eee \en \rho \enn  \bigg ).
\end{equation*}
From the above equality, performing an integration by parts and using the identity
$\boldsymbol{\nabla} ( \uu \cdot \vv ) = \boldsymbol{\nabla} \uu \cdot \vv + \boldsymbol{\nabla} \vv \cdot \uu,$ for any $\uu$ and $\vv$ in $\Haach1$, we derive
\begin{equation*}\begin{aligned}
  2 \dt ( \boldsymbol{R}\enn_2 ,\ee\enn)  = &- \dt \bigg (\io \rho\enn\eee\en \cdot\Bigl(\boldsymbol{\nabla}\uu(\tnn)\cdot\ee\enn + \boldsymbol{\nabla}\ee\enn\cdot \uu(\tnn)\Bigr) \\
& + \io \varepsilon \enn\uu(\tn)\cdot\Bigl(\boldsymbol{\nabla}\uu(\tnn)\cdot \ee\enn + \boldsymbol{\nabla}\ee \enn\cdot \uu ( \tnn )\Bigr) \\
&  + \io \rho (\tnn)\delta\uu(\tnn)\cdot\Bigl(\boldsymbol{\nabla}\uu(\tnn)\cdot\ee\enn+\boldsymbol{\nabla}\ee\enn\cdot\uu(\tnn)\Bigr)\bigg).
\end{aligned} 
\end{equation*}
Proceeding as in the previous steps, we deduce
\begin{equation*}
\begin{aligned}
  & 2 \dt (\boldsymbol{R}\enn_2,\ee\enn) \leq \frac{\mu_1 C_P }{2} \dt \Vert \ee \enn \Vert _1 ^2  \\      
  & \qquad+ \frac{3\dt}{\mu_1 C_P} \Bigl( \Vert \uu \Vert _{L^{\infty} (0,T,L^{\infty}(\Omega))} ^2 + C_S \Vert\boldsymbol{\nabla} \uu \Vert _{L^{\infty} (0,T,L^3(\Omega))} ^2 \Bigr) \\
  & \qquad\phantom{\frac{+ \dt 5(\rho_2 - \rho_1 )^2}{4 \mu_1 C_P}}\times\Bigl( \rho_2^2 \Vert\eee\en\Vert^2 + \Vert \uu \Vert _{L^{\infty} (0,T,L^{\infty}(\Omega))}^2 \Vert \varepsilon \enn \Vert ^2\Bigr)  \\
  & \qquad+   \frac{3\rho_2^2}{\mu_1 C_P} \delta t^2 \Bigl( \Vert \uu \Vert _{L^{\infty} (0,T,L^{\infty}(\Omega))} ^2 + C_S \Vert\boldsymbol{\nabla} \uu \Vert _{L^{\infty} (0,T,L^3(\Omega))} ^2 \Bigr) \int_{\tn}^{\tnn} \Vert\partial_t\uu(s)\Vert^2 ds. \\
\end{aligned}
\end{equation*}
\textit{Step 6.} Using \eqref{CalculBaseStability}, we have
\begin{equation*} \begin{aligned}
  & 2 \dt (\boldsymbol{R}^{n+1}_3 + \boldsymbol{R}^{n+1}_4 , \ee \enn )  = - 2 \dt \tilde{b} \Big ( \rho (\tnn) \delta \uu ( \tnn ), \uu ( \tnn ) , \ee \enn \Big ) \\
  & \qquad- 2 \dt \tilde{b} \big ( \rho \enn \eee \en , \uu ( \tnn ) , \ee \enn \big ) - 2 \dt \tilde{b} \big ( \varepsilon \enn \uu (\tn) , \uu ( \tnn) , \ee \enn \big ).
\end{aligned} 
\end{equation*}
Proceeding as above, we obtain
\begin{equation*} \begin{aligned}
& 2\dt (\boldsymbol{R}^{n+1}_3 + \boldsymbol{R}^{n+1}_4 ,\ee\enn) \leq \dt \frac{\mu_1 C_P}{2} \Vert \ee \enn \Vert _1 ^2\\
& +\dt \frac{3 \rho_2^2}{\mu_1 C_P \rho _1 ^2} \big( \Vert \uu \Vert^2_{L^\infty(0,T,L^\infty (\Omega))} + 9C_S \Vert \boldsymbol{\nabla} \uu \Vert^2_{L^\infty(0,T,L^3(\Omega))}\big )  \Vert \sigma \en \ee \en \Vert ^2  \\
& + \dt \frac{3}{\mu_1 C_P} \Vert\uu\Vert_{L^\infty(0,T,L^\infty (\Omega))}^2 \big( \Vert \uu \Vert^2_{L^\infty(0,T,L^\infty (\Omega))} +9C_S \Vert \boldsymbol{\nabla} \uu \Vert^2_{L^\infty(0,T,L^3(\Omega))}\big ) \Vert \varepsilon \enn \Vert ^2 \\
& + \dt ^2 \frac{3 \rho_2 ^2}{\mu_1 C_P} \big( \Vert \uu \Vert^2_{L^\infty(0,T,L^\infty (\Omega))} + 9C_S \Vert \boldsymbol{\nabla} \uu \Vert^2_{L^\infty(0,T,L^3(\Omega))}\big ) \int_{\tn} ^{\tnn} \Vert \partial _t \uu (s) \Vert ^2 ds.  \\
\end{aligned} 
\end{equation*}
\textit{Step 7.} The final estimate is obtained by applying the discrete Gronwall Lemma~\ref{lem:LGdiscret}. Let us first consider the particular case $n=0$.
Using hypothesis \eqref{HypErrPressIni} and the previous error equations, we deduce
\begin{equation*}
  \Vert \ee ^1 \Vert ^2 + \dt \Vert \ee ^1 \Vert _1 ^2 + \dt \theta \Vert s^1 \Vert ^2 + \dt^2 \Vert \boldsymbol{\nabla} r ^1 \Vert ^2 \leq C \dt (\theta + \dt).
\end{equation*}
General case $(n \geq 1 )$. Now, using steps $1$ to $7$, we finally obtain
\begin{equation*}\begin{aligned}
  & \Vert \sigma\enn \ee\enn \Vert^2 - \Vert \sigma\en \ee\en \Vert^2 + \frac{\dt^2}{\rho_1} \Bigl( \Vert \boldsymbol{\nabla}r\enn \Vert^2 - \Vert \boldsymbol{\nabla} r\en \Vert^2 \Bigr) \\
  & \qquad + \frac{2\theta\dt}{r} \Bigl( \Vert \sss\enn \Vert^2 - \Vert \sss\en \Vert^2 \Bigr) + \frac 32 \bigl( \mu_1C_P - 4r\alpha^2_2 \bigr) \dt \Vert \ee\enn \Vert^2_1 \\
  & \leq \frac{\dt^3}{\rho_1} \Vert \boldsymbol{\nabla}r\en \Vert^2 + \frac{ 2\theta\dt^2}{r} \Vert \sss\enn \Vert^2  + \Bigl( 1 + \frac{6\rho_2^2}{\mu_1C_P} K(\uu) \Bigr) \frac{\dt}{\rho_1^2}\Vert \sigma\en \ee\en \Vert^2 \\
  & + \frac{2}{\mu_1C_P} \Bigl( 3\Vert\uu\Vert_{L^\infty(0,T,L^\infty (\Omega))}^2 K(\uu) + \Vert \alpha^\prime \Vert_\infty^2 + 4 \Vert \mu^\prime\Vert_\infty^2 \Vert\DD\uu\Vert_{L^{\infty}(0,T,L^{\infty}(\Omega))}^2\Bigr) \dt \Vert \varepsilon\enn \Vert ^2 \\
  & + \frac{6 C_{S}}{\mu_1 C_P} \Vert \partial_t \uu \Vert _{L^{\infty} (0,T,L^3(\Omega))} ^2  \dt \Vert \varepsilon ^n \Vert ^2 + \frac{6\rho_2^2}{\mu_1 C_P} K(\uu)\dt^2 \int_{\tn}^{\tnn} \Vert \partial_t \uu(s) \Vert^2 ds \\
  & + \frac 2r (3 \theta \dt +1) \theta\dt \int_{\tn} ^{\tnn} \Vert \partial_t \Sig(s) \Vert^2 ds + \Bigl(1+\frac{2}{\rho_1}(1+2\dt)\Bigr) \dt^2 \int_{\tnm}^{\tnn} \Vert \partial_t p(s) \Vert^2_1 ds  \\
  & + \frac{6 C_{S}}{\mu_1 C_P} \dt^2 \Bigl( \rho^2_2 \int_{t_n}^{\tnn} \Vert \partial_{tt}\uu(s) \Vert^2_{L^{6/5}(\Omega)} ds +  \Vert \partial_{t} \uu \Vert^2_{L^{\infty}(0,T,L^3(\Omega))} \int_{t_n}^{\tnn} \Vert \partial_{t}\rho(s)\Vert^2 ds \Bigr),
\end{aligned}
\end{equation*}
with $K(\uu)=\big( \Vert \uu \Vert^2_{L^\infty(0,T,L^\infty (\Omega))} + 9C_S \Vert \boldsymbol{\nabla} \uu \Vert^2_{L^\infty(0,T,L^3(\Omega))}\big)$.
Summing the last equation over $n$ from $1$ to $k$ with $0 \leq k \leq N-1$, using $6r \alpha_2 ^2 \leq \frac{\mu_1 C_P}{2}$, Theorem~\ref{thm:ErreurPhi},
the discrete Gronwall Lemma~\ref{lem:LGdiscret} and the case $n=0$, we deduce, for all $0 \leq k \leq N-1$,
\begin{equation} \label{eqFinal}
  \begin{aligned}
    \Vert \sigma ^{k+1} \ee^{k+1} \Vert^2 + \frac{\dt^2}{ \rho_1} \Vert \boldsymbol{\nabla} r^{k+1} \Vert ^2  & + \frac{2\theta\dt}r \Vert s^{k+1} \Vert ^2 \\
    & + \dt \mu_1 C_P \sum_{n=0} ^{N-1} \Vert \ee \enn \Vert ^2 _1   \leq C \dt ( \dt + \theta ),
  \end{aligned}
\end{equation}
which concludes the proof of Theorem~\ref{thm:ErreurFinal}.
\end{proof}
\begin{rem}\label{rem:PressureBound}
From the inequality~\eqref{eqFinal}, we deduce that, for all $n$ such that $ 0 \leq n \leq N-1, $
$$ \Vert \boldsymbol{\nabla} r^{n+1} \Vert ^2 \leq C \frac{ \dt + \theta}{ \dt}.  $$
Taking $ \theta = \dt $, we obtain, for all $n$ such that $ 0 \leq n \leq N-1 $, $$ \Vert \boldsymbol{\nabla} p^{n+1} \Vert \leq C ,  $$  where $C$ is independent of $\dt$.
\end{rem}

\section{Conclusion}
We have proposed and analyzed in this paper a new scheme for the temporal discretization of incompressible equations describing the motion of Bingham fluids with variable density, plastic viscosity
 and yield stress.
A fractional time-stepping algorithm is coupled with a projection formulation for the definition of the plastic tensor. This approach, also used in the Uzawa-like algorithm for solving
viscoplastic flows, allows to handle the non-differentiable definition of the stress tensor in the Bingham constitutive laws. 
The first sub-time step of the bi-projection scheme consists in computing a non-solenoidal velocity field and is followed by a projection step so that the final velocity field is divergence free.
In classical projection methods, the Helmholtz decomposition is invoked for this projection resulting in a Poisson equation satisfied by the pseudo-pressure. Extending this approach
to density-variable flows will lead to a second-order elliptic equation with variable coefficients for the pressure. Another approach based on the interpretation of projection methods
in terms of a penalty method, Guermond and Salgado in~\cite{GuermondSalgado4,GuermondSalgado5} derived a fractional time-stepping scheme where the computation of the pressure is
achieved by solving one standard Poisson equation per time step. Unlike in~\cite{GuermondSalgado4,GuermondSalgado5}, the divergence-free velocity is used as convective velocity in the
mass conservation equation. This point is of major importance as it ensures a maximum principle for the discrete continuity equation. Lower and upper bounds on the density
are obtained.

The plastic part of the stress tensor is treated implicitly into the first sub-step (prediction step) of the fractional time-stepping method. As in the Uzawa-like method for Bingham flows,
a fixed-point algorithm is used to compute the plastic tensor.
Due to a pseudo-time relaxation term added in the Bingham projection operator, the convergence is geometric with common ratio $(1-\theta)$ where $\theta$ is a prescribed relaxation parameter.
The main results derived in this paper are the unconditional stability and the convergence of the bi-projection scheme. More precisely, errors committed by approximating the velocity and 
the density are bounded from above by a term of the order of $\sqrt{\delta t(\delta t+\theta)}$. Hence, by choosing $\theta$ of the order of the time step ensures the error of the bi-projection scheme to be of first order.

\section*{Acknowledgements}
This research was financed by the French Government Laboratory of Excellence initiative n$^\textrm{o}$ANR-10-LABX-0006, by the French
National Research Agency (ANR) RAVEX project, and by the French National Center for Scientific Research (CNRS) TelluS project.
This is Laboratory of Excellence ClerVolc contribution number $302$.

\bibliographystyle{plain}
\bibliography{bibfi.bib}

\end{document}